\documentclass[reqno]{amsart}

\usepackage{amsfonts}
\usepackage{amssymb}
\usepackage{mathabx}
\usepackage{amscd}
\usepackage{pictexwd,dcpic}
\usepackage{graphicx}
\usepackage{xcolor}
\usepackage{tikz}

\usepackage{soul}

\usepackage{enumitem}
\usepackage{fullpage}
\usepackage{hyperref}
\hypersetup{
    colorlinks=true,
    linkcolor=blue,
    filecolor=magenta,
    urlcolor=cyan,
    citecolor=cyan,}
\usepackage[nameinlink,capitalize]{cleveref}

\definecolor{darkgreen}{RGB}{18, 80, 10}
\definecolor{gold}{rgb}{0.85,.66,0}

\title{On Rees algebras of ideals and modules with\\ weak residual conditions}
\author{Alessandra Costantini}
\address{Department of Mathematics, Tulane University, Gibson Hall, New Orleans LA 70118}
\email{acostantini@tulane.edu}
\author{Edward F. Price III}
\address{Department of Mathematics and Computer Science, Colorado College, Tutt Science Center, Colorado Springs \indent CO 80903}
\email{eprice@coloradocollege.edu}
\author{Matthew Weaver}
\address{School of Mathematical and Statistical Sciences, Arizona State University, Wexler Hall, Tempe AZ 85281}
\email{matthew.j.weaver@asu.edu}

\date{}	

\newtheorem{thmx}{Theorem}

\newtheorem{thm}{Theorem}[section]

\newtheorem{prop}[thm]{Proposition}
\newtheorem{lemma}[thm]{Lemma}
\newtheorem{cor}[thm]{Corollary}
\numberwithin{equation}{section}

\theoremstyle{definition}
\newtheorem{rem}[thm]{Remark}
\newtheorem{set}[thm]{Setting}
\newtheorem{notat}[thm]{Notation}
\newtheorem{defn}[thm]{Definition}
\newtheorem{quest}[thm]{Question}
\newtheorem{ex}[thm]{Example}

\DeclareRobustCommand\longtwoheadrightarrow{\relbar\joinrel\twoheadrightarrow}

\def\a{\mathfrak{a}}
\def\A{\mathcal{A}}
\def\J{\mathcal{J}}
\def\K{\mathcal{K}}
\def\L{\mathcal{L}}
\def\m{\mathfrak{m}}
\def\p{\mathfrak{p}}
\def\q{\mathfrak{q}}
\def\R{\mathcal{R}}
\def\S{\mathcal{S}}
\def\F{\mathcal{F}}
\def\D{\mathcal{D}}

\def\ann{\mathop{\rm ann}}
\def\ass{\mathop{\rm Ass}}
\def\MinPrimes{\mathop{\rm Min}} 
\def\dim{\mathop{\rm dim}}
\def\fitt{\mathop{\rm Fitt}}

\def\hgt{\mathop{\rm ht}}

\def\spec{\mathop{\rm Spec}}

\def\rank{\mathop{\rm rank}}

\def\bideg{\mathop{\rm bideg}}

\def\ker{\mathop{\rm ker}}
\def\coker{\mathop{\rm coker}}

\begin{document}

\begin{abstract}
Let $E$ be a module of projective dimension one over $R=k[x_1,\ldots,x_d]$. If $E$ is presented by a matrix $\varphi$ with linear entries and the number of generators of $E$ is bounded locally up to codimension $d-1$, the Rees ring $\R(E)$ is well understood. In this paper, we study $\R(E)$ when this generation condition holds only up to codimension $s-1$, for some $s<d$. Moreover, we provide a generating set for the ideal defining this algebra by employing a method of successive approximations of the Rees ring. Although we employ techniques regarding Rees rings of modules, our findings recover and extend known results for Rees algebras of perfect ideals with grade two in the case that $\rank E=1$.
\end{abstract}

\maketitle

 
\section{Introduction}

For $I=(f_1,\ldots,f_n)$ an ideal in a Noetherian ring $R$, the Rees ring of $I$ is the subring $\R(I) = R[It]= R[f_1t,\ldots, f_n t]$ of the polynomial ring $R[t]$, for $t$ an indeterminate. As a graded algebra, one also has that $\R(I) = R\oplus It \oplus I^2t^2\oplus\cdots$. From this latter description, one sees that $\R(I)$ carries information on every power $I^j$ and their asymptotic behavior for large exponents $j \gg 0$. As such, the Rees ring and its associated algebras have proven to be indispensable within the study of reductions and various multiplicities. In the geometric setting, the Rees ring $\R(I)$ is often called the \textit{blowup algebra}, as $\mathrm{Proj}(\R(I))$ is precisely the blowup of $\spec(R)$ along the subscheme $V(I)$. Alternatively, if $R=k[x_1,\ldots,x_d]$ and the polynomials $f_1,\ldots,f_n$ are homogeneous of a common degree, the Rees algebra serves as the coordinate ring of $\mathrm{graph}(\Phi)$ where $\Phi:\, \mathbb{P}_k^{d-1} \dashrightarrow \mathbb{P}_k^{n-1}$ is the rational map defined by $f_1,\ldots,f_n$. We may also extend this notion to Rees algebras of \textit{modules} in order to treat the case of repeated or successive blowups, or also compositions of rational maps between projective varieties. Indeed, the successive blowup of an affine scheme along disjoint subschemes $V(I)$ and $V(J)$ corresponds to the Rees ring $\R(I\oplus J)$. 

For an $R$-module $E$, the Rees ring $\R(E)$ is defined as $\R(E) = \S(E)/\tau(\S(E))$ where $\S(E)$ is the symmetric algebra of $E$ and $\tau(\S(E))$ is its $R$-torsion submodule. Although seemingly different, this recovers the previous notion when $E$ is actually an $R$-ideal. Since identifying the torsion of $\S(E)$ is seldom a simple task, it is typically more advantageous to describe $\R(E)$ as a quotient of a polynomial ring $R[T_1,\ldots,T_n]$. The ideal $\J$ defining this quotient is aptly called the \textit{defining ideal} of $\R(E)$ and its generators are called the \textit{defining equations} of $\R(E)$.

In general, determining the equations of $\J$ is an arduous task, however there has been notable success for Rees rings of perfect ideals with small codimension \cite{BM16,CPW,DRS18,KPU11,KPU17,Morey96,MU96,Nguyen14,Nguyen17,Weaver23,Weaver24} and Rees algebras of modules with small projective dimension \cite{Costantini,CPW,Madsen,SUV03,Weaver23,Weaver24Preprint} in a multitude of settings. Many of these results share common assumptions, most notably the \textit{residual condition} $G_d$, where $d=\dim R$. This condition was introduced by Artin and Nagata in \cite{AN72}, to study residual intersections of an ideal (see \Cref{sec:preliminaries}). For $E$ a module with rank $e$, one says that $E$ satisfies the condition $G_s$ if $\mu(E_{\p}) \leq \dim R_{\p} +e -1$ for every $\p \in \spec(R)$ with $1\leq \dim R_{\p} \leq s-1$. Additionally, $E$ satisfies $G_\infty$ if $E$ satisfies $G_s$ for all $s$. Here $\mu(E_{\p})$ denotes the minimal number of generators of $E_\p$. 

Although seemingly unassuming, the condition $G_d$ where $d=\dim R$ is quite powerful within the study of Rees algebras, as it often dictates the prime ideals of $R$ at which $\R(E)$ and $\S(E)$ coincide locally. As such, this assumption can seldom be weakened when comparing these two algebras. However, there are many classes of ideals and modules that do not satisfy this condition, yet have notable Rees rings. For example, let $\Omega_{R/k}$ denote the module of K\"ahler differentials of a complete intersection ring $R=k[x_1,\ldots,x_n]/I$ with $\dim R =d \geq 2$. The Rees algebra $\R(\Omega_{R/k})$ and its related rings are called \textit{tangent algebras} due to their connections with tangential varieties arising in algebraic geometry, and have been well studied \cite{SUV12}. However, the module $\Omega_{R/k}$ often fails to satisfy $G_d$, for instance if $R$ is not a normal ring \cite[Sec. 6]{Weaver24Preprint}. Additionally, classes of examples of grade-two perfect ideals that do not satisfy $G_d$ can easily be found, for instance in \cite{CDG22,DRS18}.

While the condition $G_d$ is essential for most arguments, there has been success in determining the defining ideal of Rees rings of perfect ideals with grade two satisfying $G_{d-1}$ \cite{DRS18,Nguyen14,Nguyen17}. Much of this work was extended more generally to modules of projective dimension one satisfying $G_{d-1}$ in our recent paper \cite{CPW}. In each of these instances, the weaker residual condition is supplemented with a strict rank condition on a presentation matrix $\varphi$, which appears difficult to relax. The objective of this paper is to further extend this previous work to modules of projective dimension one satisfying $G_{s}$ for \textit{any} $s<d$. Our main results \Cref{Defining Ideal Column Case} and \Cref{Defining Ideal Row Case} are summarized as follows. 

\begin{thmx}\label{Intro theorem - Defining Ideal of R(E)}
Let $R=k[x_1,\ldots,x_d]$ be a polynomial ring with $d\geq 3$ over a field $k$, and let $E$ be a $R$-module of projective dimension one and $\rank E=e$. Assume that $E$ is minimally generated by $\mu(E) = n\geq d+e$ elements and $E$ satisfies $G_{s}$, but not $G_{s+1}$, for some integer $2 \leq s \leq d-1$. Assume furthermore that $E$ has a presentation matrix $\varphi$ consisting of linear entries in $R$ with $I_1(\varphi) = (x_1,\ldots,x_d)$ and that, after possibly a change of coordinates, modulo $(x_1, \ldots, x_{s})$ the matrix $\overline{\varphi}$ has rank 1.

\begin{enumerate}
    \item[{\rm(a)}] If the nonzero entries of $\overline{\varphi}$ are in one column, the Rees algebra of $E$ is $\R(E) \cong R[T_1,\ldots,T_n]/\J$ with
$$\J= (\ell_1,\ldots,\ell_{n-e}):(x_1,\ldots,x_s) = (\ell_1,\ldots,\ell_{n-e})+I_{s}(M)$$
where $\ell_1,\ldots,\ell_{n-e}$ are linear forms such that $[\ell_1 \ldots \ell_{n-e}]= [T_1 \ldots T_n] \cdot \varphi$ and $M$ is an $s\times (n-e-1)$ matrix with entries in $k[T_1, \ldots, T_n]$ such that $[\ell_1\ldots\ell_{n-e-1}] = [x_1\ldots x_{s}]\cdot M$.

\vspace{1mm}

    \item[{\rm(b)}] If the nonzero entries of $\overline{\varphi}$ are in one row, the Rees algebra of $E$ is $\R(E) \cong R[T_1,\ldots,T_n]/\J$ with
$$\J= (\ell_1,\ldots,\ell_{n-e}):(x_1,\ldots,x_s) = (\ell_1,\ldots,\ell_{n-e})+I_{s}(M)+I_{s+1}(C)$$
where $\ell_1,\ldots,\ell_{n-e}$ are linear forms such that $[\ell_1 \ldots \ell_{n-e}]= [T_1 \ldots T_n] \cdot \varphi$. Moreover, $M$ is an $s\times (n-e-d+s)$ matrix with entries in $k[T_1, \ldots, T_n]$ such that $[\ell_1\ldots\ell_{n-e-d+s}] = [x_1\ldots x_{s}]\cdot M$ and $C$ is an $(s+1)\times (n-e)$ matrix such that $[\ell_1 \ldots \ell_{n-e}] = [x_1 \ldots x_s \,\,\, T_n] \cdot C$.
\end{enumerate}
Moreover, in either case, the Rees algebra $\R(E)$ is a Cohen-Macaulay domain of dimension $d+e$.
\end{thmx}

With the condition that $\rank \overline{\varphi}=1$, the nonzero entries of $\overline{\varphi}$ are concentrated in either a single row or a single column. Interestingly,  $\R(E)$ is Cohen-Macaulay in either case, but the shape of the defining ideal differs in the two settings. However, in either case, the matrices $M$ and $C$ which provide the nontrivial equations are obtained from submatrices of the \textit{Jacobian dual} matrix of the presentation $\varphi$ (see \Cref{sec:preliminaries}).

Not only does the Rees ring $\R(E)$ differ in each of the cases above, but the \textit{special fiber ring} $\F(E) \coloneq \R(E) \otimes_R k$ does as well. In particular, the \textit{analytic spread} of $E$, i.e. the Krull dimension $\ell(E) \coloneq \dim \F(E)$ differs in each case. Indeed, as a consequence of \Cref{Intro theorem - Defining Ideal of R(E)}, we obtain the result below, which is a compilation of \Cref{F(I)defideal Column} and \Cref{F(I)defideal Row}.

\begin{thmx}\label{Intro theorem - Defining Ideal of F(E)}
With the assumptions of \Cref{Intro theorem - Defining Ideal of R(E)}, we have the following. 
\begin{enumerate}
    \item[{\rm(a)}] If the nonzero entries of $\overline{\varphi}$ are in one column, the special fiber ring is $\F(E) \cong k[T_1,\ldots,T_n]/I_{s}(M)$. Moreover, $\F(E)$ is a Cohen-Macaulay domain of dimension $\ell(E)= s+e$.

    \vspace{1mm}
   \item[{\rm(b)}] If the nonzero entries of $\overline{\varphi}$ are in one row, the special fiber ring is $\F(E) \cong k[T_1,\ldots,T_n]/I_{s}(M)$. Moreover, $\F(E)$ is a Cohen-Macaulay domain of dimension $\ell(E)= d+e-1$. \end{enumerate}
\end{thmx}

We note that the Cohen-Macaulay property of $\F(E)$ does not depend on the two cases of \Cref{Intro theorem - Defining Ideal of F(E)}, but the analytic spread of $E$ does. However, the two phenomena do coincide when $s=d-1$, recovering our findings in \cite{CPW}. Additionally, when $\rank E=1$, the module $E$ is isomorphic to an $R$-ideal $I$, hence \Cref{Intro theorem - Defining Ideal of R(E)} and \Cref{Intro theorem - Defining Ideal of F(E)} are applicable to Rees algebras of perfect ideals of grade two. In particular, when $e=1$ and $s=d-1$, \Cref{Intro theorem - Defining Ideal of R(E)} and \Cref{Intro theorem - Defining Ideal of F(E)} recover the main results of \cite{Nguyen14,Nguyen17}. Moreover, to the best of our knowledge, our findings provide the first known results for Rees algebras of ideals and modules satisfying $G_s$ for $s\leq d-2$.

In \cite{CPW}, the primary technique to study the Rees ring was to produce a \textit{generic Bourbaki ideal} $I$ of $E$, reducing the study of $\R(E)$ to the study of $\R(I)$, where information is more accessible. However, in the present paper, we forgo this technique and instead employ a method of \textit{successive approximations} of the Rees algebra $\R(E)$, as the essential ingredient to the proofs of \Cref{Intro theorem - Defining Ideal of R(E)} and \Cref{Intro theorem - Defining Ideal of F(E)}. From a free presentation $R^{m} \overset{\varphi}{\rightarrow}R^n \rightarrow E \rightarrow 0$, one factors the map  $R^n \rightarrow E$ into a sequence of epimorphisms 
$$R^n = E_{m}\twoheadrightarrow E_{m-1} \twoheadrightarrow \cdots\cdots \twoheadrightarrow E_{0} = E$$ 
which induces a sequence of epimorphisms 
$$R[T_1,\ldots,T_n] = \R(E_{m})\twoheadrightarrow \R(E_{m-1}) \twoheadrightarrow \cdots\cdots \twoheadrightarrow \R(E_{0}) = \R(E)$$
factoring the natural map $\pi:\,R[T_1,\ldots,T_n] \twoheadrightarrow \R(E)$. If one is able to understand these intermediate algebras, and their defining ideals in particular, then one has a much better chance of determining the defining ideal of $\R(E)$. We remark that this is not an entirely novel approach, and this technique originates in \cite{Madsen}. Moreover, we note that variations of this technique have been previously applied in \cite{BM16,KPU11,Weaver23,Weaver24}. However, our main contribution is to show that the construction of these approximation modules is compatible with the residual condition $G_s$, when $E$ has projective dimension one. Moreover, we describe how the shape of the defining ideal of an intermediate ring above can be deduced from the previous algebra.
\smallskip

We now briefly describe how this paper is organized. In \Cref{sec:preliminaries}, we review the necessary preliminary material on Rees rings of modules and their associated algebras, as well as residual intersections of ideals. We also introduce the method of successive approximations of the Rees algebra $\R(E)$ and how to implement this construction in the study of the defining ideal. In \Cref{sec: Two forms of the matrix}, we discuss the primary setting for the paper and identify the two cases that must be considered, stemming from the rank condition in \Cref{Intro theorem - Defining Ideal of R(E)}. The remainder of the section is spent discussing the similarities between the two cases, while the proceeding sections are spent analyzing their differences. In \Cref{sec:column section,sec:row section}, we determine the defining ideal $\J$ of the Rees ring $\R(E)$ in each of the two cases. With this, the Cohen-Macaulay and fiber-type property are explored in each case. Finally, in \Cref{sec: Examples} we include several examples showing that the conclusion of \Cref{Intro theorem - Defining Ideal of R(E)} may not necessarily hold if we weaken any of the assumptions.


\section{Preliminaries} \label{sec:preliminaries}

In this section, we recall background information on modules with a rank and their Rees algebras, needed throughout the article. In particular, we include crucial observations regarding modules of projective dimension one in \Cref{sec:column stripping}. For the present setting, assume that $R$ is a Noetherian ring and $E$ is a finitely generated $R$-module. We note that an ideal of positive grade is isomorphic to a torsion-free module of rank one, hence all statements discussed here apply to Rees rings of ideals when $\rank E=1$.


\subsection{Modules and their Rees algebras}

Recall that $E$ is said to have a \textit{rank} if $E \otimes_R \mathrm{Quot}(R) \cong \mathrm{Quot}(R)^e$ for some integer $e$, where $\mathrm{Quot}(R)$ denotes the total ring of quotients of $R$; in this case, we write $\rank(M)=e$. In particular, free modules have a rank, hence modules of finite projective dimension do as well, since rank is additive along exact sequences.

Given a generating set $E = Ra_1+\cdots +Ra_n$, such a module admits a corresponding free presentation 
\begin{equation}\label{GeneralPresentation}
R^m \overset{\varphi}{\longrightarrow} R^n \longrightarrow {E} \longrightarrow 0.
\end{equation}
Recall that, for $0 \leq i \leq n$, the $i^{\text{th}}$ \textit{Fitting ideal} of $E$ is the ideal $\fitt_i(E) \coloneq I_{n-i}(\varphi)$, i.e. the ideal generated by all $(n-i)\times (n-i)$ minors of $\varphi$. We remark that these ideals are particularly useful invariants of the module $E$.

\begin{prop}[{\cite[20.4--20.6]{Eisenbud}}]\label{fitt}
 With a module $E$ as above, we have the following.
 \begin{itemize}
     \item[{\rm(a)}] $\fitt_i(E)$ depends only on the module $E$ and the index $i$. In particular, it is independent of the choice of presentation. 
     \item[{\rm(b)}] If $R$ is local, then $\fitt_i(E) = R$ if and only if $\mu(E) \leq i$.
     \item[{\rm(c)}] $V(\fitt_i(E)) = \{\p \in \spec(R)\,|\, \mu(E_\p)\geq i+1\}$.
 \end{itemize}
\end{prop}

Condition (c) above is particularly useful as it easily relates to the condition $G_s$ from the introduction.

\begin{rem}\label{defGs}
A module $E$ with $\rank E =e$ satisfies $G_s$ if and only if $\hgt \fitt_i(E) \geq i-e+2$ for all $e\leq i\leq s+e-2$. As before, we say that $E$ satisfies $G_\infty$ if $E$ satisfies $G_s$ for all $s$.
\end{rem}

In addition to the Fitting invariants, a free presentation of $E$ also provides information on the \textit{symmetric algebra} of $E$, $\S(E)$. More precisely, if $E = R a_1 + \ldots + R a_n$ and (\ref{GeneralPresentation}) is a presentation corresponding to this generating set, there is a natural homogeneous epimorphism of graded $R$-algebras 
\begin{equation}\label{Sym natural map}
\eta \colon R[T_1,\ldots,T_n] \longtwoheadrightarrow \S(E)    
\end{equation}
defined by mapping each $T_i$ to $a_i\in [\S(E)]_1$, and extending $R$-linearly. Moreover, the kernel $\L= \ker(\eta)$ is generated by linear forms $\ell_1,\ldots,\ell_m$ such that $[T_1\ldots T_n]\cdot \varphi = [\ell_1 \ldots \ell_m]$. Thus, there is an induced isomorphism $\S(E) \cong R[T_1,\ldots,T_n]/\L$. We remark that this description is independent of the choice of presentation (\ref{GeneralPresentation}). 

As noted in the introduction, the \textit{Rees algebra} of $E$ is $\R(E) = \S(E)/\tau(\S(E))$, where $\tau(\S(E))$ is the $R$-torsion submodule of $\S(E)$. If $E\cong I$ for an $R$-ideal $I$, the Rees algebra $\mathcal{R}(E)$ is isomorphic to the subalgebra $\R(I) = R[It] \subseteq R[t]$ defined in the introduction. If $E$ is an $R$-module with rank $e$, the Krull dimension of $\R(E)$ is $\dim \R(E)=d+e$, where $d=\dim R$ (see, e.g., \cite[2.2]{SUV03}). 

Composing the map $\eta$ in (\ref{Sym natural map}) with the natural map factoring $\tau(\S(E))$, one obtains a second homogeneous epimorphism 
\begin{equation}\label{Rees natural map}
\pi \colon R[T_1,\ldots,T_n] \longtwoheadrightarrow \R(E)    
\end{equation}
mapping $T_i \mapsto a_i\in [\R(E)]_1$. There is an induced isomorphism $\R(E) \cong R[T_1,\ldots,T_n]/\J$, where $\J \coloneq \mathrm{ker}(\pi)$ is the \textit{defining ideal} of $\R(E)$. By construction, it is clear that $\L\subseteq \J$. Moreover, $\L$ is actually the degree-one component of $\J$, $\L= [\J]_1$ \cite{Vasconcelos91}. One says that $E$ is of \textit{linear type} if $\L = \J$, as in this case all generators of $\J$ are linear forms.

If $E$ is not of linear type, the higher-degree generators of $\J$ can often be detected by means of a Jacobian dual matrix of the presentation $\varphi$ of $E$, or via the special fiber ring $\mathcal{F}(E)$ of $E$. We next briefly recall both of these notions.

\begin{defn}
Let $E$ be a module with presentation (\ref{GeneralPresentation}) and let $\ell_1,\ldots,\ell_m$ be the generators of $\L$ as before. There exists an $r\times m$ matrix $B(\varphi)$ with linear entries in $R[T_1,\ldots,T_n]$ such that
\begin{equation}\label{JD equation}
[T_1 \ldots T_n] \cdot \varphi=[\ell_1 \ldots \ell_m] = [x_1\ldots x_r]\cdot B(\varphi),
\end{equation}
where $(x_1,\ldots,x_r)$ is an ideal containing $I_1(\varphi)$. The matrix $B(\varphi)$ is a \textit{Jacobian dual} of $\varphi$, with respect to the sequence $x_1,\ldots,x_r$.    
\end{defn}

We remark that $B(\varphi)$ is not unique in general. However, if $R=k[x_1,\ldots,x_d]$, then there is a unique Jacobian dual matrix $B(\varphi)$, with respect to $x_1,\ldots,x_d$, if and only if the entries of $\varphi$ are linear \cite[p.~47]{SUV93}. In this case, the entries of $B(\varphi)$ belong to the subring $k[T_1,\ldots,T_n]$, a fact which we will exploit in several proofs in later sections of this article.

\begin{defn}
Assume that $R$ is a local ring with maximal ideal $\m$ and residue field $k$. Let $E$ be a finitely generated $R$-module as above. The \textit{special fiber ring} of $E$ is 
$$\F(E) \coloneq  \R(E)\otimes_R k \cong \R(E)/\m \R(E).$$ 
The Krull dimension of $\F(E)$ is called the \textit{analytic spread} of $E$ and is denoted by $\ell(E)$.    
\end{defn}

This notion may also be adapted to the graded setting when $R$ is a standard graded polynomial ring $R=k[x_1, \ldots, x_d]$, over a field $k$. In this setting, the special fiber ring is similarly defined using the unique homogeneous maximal ideal $\m = (x_1,\ldots,x_d)$.

The natural map $\pi$ of (\ref{Rees natural map}) induces an epimorphism of graded $k$-algebras 
$$\psi \colon R[T_1,\ldots,T_n] \otimes_R k \cong k[T_1,\ldots,T_n] \longtwoheadrightarrow \F(E) \cong \R(E)/ \m \R(E)$$
such that $\psi(T_i) = a_i + \m \R(E) \in [\F(E)]_1$. In particular, $\ker(\psi)R[T_1, \ldots, T_n] + \L \subseteq \J$ and we say that $E$ is of \textit{fiber type} if this containment is an equality. 


\subsection{Successive approximations of Rees algebras}\label{sec:column stripping}

In this subsection, we discuss a useful technique used to approximate Rees algebras of modules of projective dimension one, expanding on the method introduced in \cite{Madsen}. Our setting throughout is as follows.

\begin{set}\label{Strip Setting}
Let $R$ be a Noetherian ring, and let $E$ be an $R$-module of rank $e$ having projective dimension one and minimally generated by $\mu(E) =n$ forms. With this, we may write
\begin{equation}\label{pdim1 pres}
0 \longrightarrow R^{n-e} \overset{\varphi}{\longrightarrow}R^n \longrightarrow E \longrightarrow 0    
\end{equation}
to denote a minimal resolution of $E$.  Let $\varphi_i$ denote the submatrix of $\varphi$ obtained by deleting the last $i$ columns of $\varphi$ and let $E_i = \coker \varphi_i$. 
\end{set}

Notice that there is a chain of epimorphisms
\begin{equation}\label{Module Surjections}
R^n = E_{n-e}\twoheadrightarrow E_{n-e-1} \twoheadrightarrow \cdots\cdots \twoheadrightarrow E_{0} = E.
\end{equation}
factoring the map $R^n\twoheadrightarrow E$ in (\ref{pdim1 pres}). With the sequence of modules above, we make the following observation.

\begin{lemma}\label{Gs Right to Left}
With the assumptions of \Cref{Strip Setting} and the modules in (\ref{Module Surjections}), we have the following.
\begin{enumerate}
    \item[{\rm (a)}] For $1\leq i\leq n-e-1$, each $E_i$ is an $R$-module of projective dimension one with $\rank E_i = e+i$.

    \item[{\rm (b)}] If $E_i$ satisfies $G_s$, then $E_j$ satisfies $G_s$ as well, for all $j\geq i$. 
\end{enumerate}
\end{lemma}

\begin{proof}
To verify (a) it suffices to show that $\varphi_i$ is an injective map for each $i$ in the given range. However, this is clear from linear algebra, as $\varphi$ is injective. Now that $\varphi_i$ is seen to be injective, the short exact sequence
\begin{equation}\label{E_i presentation}
0 \longrightarrow R^{n-e-i} \overset{\varphi_i}{\longrightarrow}R^n \longrightarrow E_i \longrightarrow 0    
\end{equation}
shows that $E_i$ has projective dimension one and $\rank E_i=e+i$, by additivity of rank.

To prove (b), it suffices to show that if $E=E_0$ satisfies $G_s$, then $E_1$ does as well. The claim then follows by induction after reindexing. Since $E$ satisfies $G_s$ and $\fitt_i(E) \subseteq \fitt_{i+1}(E_1)$, one has that $\hgt \fitt_{i+1}(E_1) \geq i-e+2$ for all $e\leq i\leq s+e-2$. After adjusting the indices, we see that $\hgt \fitt_{i}(E_1) \geq i-(e+1)+2$ for all $e+1\leq i\leq s+(e+1)-2$. Since $\rank E_1 =e+1$, it follows that $E_1$ satisfies $G_s$. 
\end{proof}

\begin{rem}
From \Cref{Gs Right to Left} it follows that if the module $E$ satisfies $G_s$, then each of the modules in (\ref{Module Surjections}) satisfy $G_s$ as well. However, it is possible that some of the modules satisfy $G_t$ for $t>s$. For instance, $E_{n-e} = R^n$ is free and hence satisfies $G_\infty$. 
\end{rem}

Notice that the sequence of epimorphisms in (\ref{Module Surjections}) induces a sequence of homogeneous epimorphisms of symmetric algebras:
\begin{equation}\label{Sym alg Surjections}
R[T_1,\ldots,T_n] = \S(E_{n-e}) \twoheadrightarrow \S(E_{n-e-1}) \twoheadrightarrow \cdots\cdots\twoheadrightarrow \S(E_{0}) = \S(E),   
\end{equation}
which factors the natural map in (\ref{Sym natural map}). Moreover, this induces a sequence of homogeneous epimorphisms of Rees algebras:
\begin{equation}\label{Rees alg Surjections}
R[T_1,\ldots,T_n] = \R(E_{n-e}) \twoheadrightarrow \R(E_{n-e-1}) \twoheadrightarrow \cdots\cdots\twoheadrightarrow \R(E_{0}) = \R(E)    
\end{equation}
by further factoring $R$-torsion, which also factors the natural map in (\ref{Rees natural map}). In particular, the ideals defining the intermediate algebras of (\ref{Sym alg Surjections}) and (\ref{Rees alg Surjections}) relate to each other. 

\begin{notat}\label{LiJinotation}
Write $\L_i$ to denote the defining ideal of $\S(E_i)$, i.e. the kernel of the epimorphism
$$R[T_1,\ldots,T_n] \longtwoheadrightarrow \S(E_i),$$
and write $\J_i$ to denote the defining ideal of $\R(E_i)$, i.e. the kernel of the epimorphism
$$R[T_1,\ldots,T_n] \longtwoheadrightarrow \R(E_i).$$  
\end{notat}

Writing $\L$ to denote the defining ideal of $\S(E)$, with the presentation of $E$ in (\ref{pdim1 pres}) note that $\L=(\ell_1,\ldots,\ell_{n-e})$ where $[\ell_1\ldots \ell_{n-e}] = [T_1\ldots T_n]\cdot \varphi$. Moreover, the defining ideal of $\S(E_i)$ is precisely $\L_i=(\ell_1,\ldots, \ell_{n-e-i})$ as $E_i$ is presented by (\ref{E_i presentation}) and $\varphi_i$ is a submatrix of $\varphi$. Hence there is a chain of inclusions
$$0=\L_{n-e} \subsetneq \L_{n-e-1}\subsetneq \cdots\cdots \subsetneq \L_1 \subsetneq \L_0 = \L$$
where the generators of the ideals $\L_i$ are well understood.

Unfortunately, the ideals $\J_i$ are not as simple. However, we have a chain of inclusions
$$0=\J_{n-e} \subsetneq \J_{n-e-1}\subsetneq \cdots\cdots \subsetneq \J_1 \subsetneq \J_0 = \J,$$
and we note that, if $R$ is a domain, for each $i$ the Rees ring $\R(E_i)$ is a domain of dimension $d+e+i$, hence each $\J_i$ is a prime ideal of height $n-e-i$. Moreover, each successive quotient $\J_{i}/\J_{i+1}$ fits into the short exact sequence
$$0\longrightarrow \J_{i}/\J_{i+1} \longrightarrow \R(E_{i+1}) \longrightarrow\R(E_i) \longrightarrow 0.$$

The quotients $\J_{i}/\J_{i+1}$ have the advantage of being simpler than the entire ideal $\J$, as they are prime $\R(E_{i+1})$-ideals of height 1, whenever $R$ is catenary. Moreover, if generating sets can be found for each $\J_{i}/\J_{i+1}$, then they can be lifted and combined to produce a generating set of $\J$.


\subsection{Residual intersections}

Lastly, we very briefly recall the notion of residual intersections of an ideal. This notion goes back to the work of Artin and Nagata \cite{AN72}, who informally also introduced the $G_s$ condition; indeed, if an ideal $I$ satisfies $G_s$, one can easily identify an optimal generating set for a residual intersection of $I$ (see, e.g., \cite[1.6]{U94}). In the present paper, we aim to describe the structure of certain quotients $\J_{i}/\J_{i+1}$ in terms of residual intersections, for which we require very few technical aspects of the subject. Hence, we omit much of the unnecessary background material, and refer the curious reader to \cite{HU88,HU90} for a more rigorous treatment of the matter.

\begin{defn}[{\cite[1.1]{HU88}}] \label{defResInt}
Let $R$ be a Cohen-Macaulay local ring and $I$ an $R$-ideal. For $s$ an integer with $s\geq \hgt I$, a proper ideal $J=\a:I$ is an \textit{$s$-residual intersection} of $I$ if $\a = (a_1,\ldots,a_s) \subseteq I$ and $\hgt J\geq s$.
\end{defn}

In particular, we restrict our attention to residual intersections of ideals generated by regular sequences. Not only do these ideals possess strong properties, their generation and resolutions are also well known \cite{BKM90}. 

\begin{thm}\label{thmResInt}
Let $R$ be a Cohen-Macaulay local ring and let $I=(x_1,\ldots,x_g)$ be an ideal with $x_1,\ldots,x_g$ an $R$-regular sequence. Let $J=\a:I$ be an $s$-residual intersection of $I$ for some $\a=(a_1,\ldots,a_s) \subseteq I$.
\begin{itemize}
    \item[{\rm(a)}] {\rm(}\cite[1.4 and 1.5]{HU90}{\rm)} $R/J$ is a Cohen-Macaulay ring.    
    \item[{\rm(b)}] {\rm(}\cite[1.5 and 1.8]{HU90}{\rm)} If $B$ is a $g\times s$ matrix with $[a_1 \ldots a_s] = [x_1\ldots x_g]\cdot B$, then $J =\a + I_g(B)$.
\end{itemize}
\end{thm}


\section{Modules with weak residual conditions}\label{sec: Two forms of the matrix}

We now begin our treatment of the Rees algebra of an $R$-module $E$ with projective dimension one satisfying $G_s$ for $s<d=\dim R$. We remark that we supplement the weakening of the condition $G_d$ with a strict rank condition on a presentation matrix $\varphi$ of $E$, similar to the settings of \cite{CPW,Nguyen17}. We also note that for $s<d-1$, this rank assumption gives rise to two cases which must be treated separately. For the duration of the paper, we adopt the following setting.

\begin{set}\label{rank1setting}
    Let $R=k[x_1,\ldots,x_d]$ be a standard-graded polynomial ring over a field $k$, with $\m = (x_1,\ldots,x_d)$ and $d\geq 3$. Let $s$ be an integer with $2 \leq s \leq d-1$. Let $E$ be a $R$-module of projective dimension one and $\rank E=e$ minimally generated by $\mu(E) = n\geq d+e$  elements,  satisfying the following assumptions:
    \begin{itemize}
        \item[(i)] The module $E$ has a presentation matrix $\varphi$ consisting of linear entries with $I_1(\varphi) = \m$.
        \item[(ii)] After a possible change of coordinates, the matrix $\varphi$ has rank 1 modulo $\p=(x_{1},\ldots,x_s)$.
        \item[(iii)] The module $E$ satisfies $G_{s}$, but not $G_{s+1}$.
    \end{itemize}
\end{set}

Condition (ii) is the aforementioned rank condition and is certainly the most strict of the assumptions above. However, we note that this condition holds automatically if $s=d-1$ and $n=d+e$ \cite[4.4]{CPW}. Additionally, notice that as $E$ satisfies $G_2$, it is torsion-free. As a consequence, modules of rank one as in \Cref{rank1setting} are isomorphic to notable ideals.

\begin{rem} \label{G2P prop}
Under the assumptions of \Cref{rank1setting}, $E$ has rank $e=1$ if and only if $E$ is isomorphic to a perfect ideal of grade two. 
\end{rem}

\begin{proof}
Let $E$ be a module as in \Cref{rank1setting} with rank $e=1$. Since $E$ satisfies $G_s$ for some $s\geq 2$, it is torsion-free, and hence isomorphic to an $R$-ideal $I$. Moreover, as $E\cong I$ has projective dimension one, from (\ref{sec 3 pdim1 presentation}) it follows that $I$ is presented by a $n\times (n-1)$ matrix $\varphi$. Moreover, as $I$ satisfies $G_2$ we have that $\hgt I_{n-1}(\varphi) = \hgt \fitt_{1}(I) \geq 2$ and the claim follows from the Hilbert-Burch theorem \cite[20.15]{Eisenbud}. 
\end{proof}

In particular, any result obtained in the present setting for Rees algebras of modules with projective dimension one can be applied to Rees algebras of perfect ideals of grade two.

The rank condition (ii) above has a profound implication on the structure of the matrix $\varphi$, as discussed in the following crucial remark. 

\begin{rem}\label{rank1shapes}
With the assumptions of \Cref{rank1setting}, the module $E$ has minimal free resolution of the form
\begin{equation}\label{sec 3 pdim1 presentation}
0 \longrightarrow R^{n-e} \overset{\varphi}{\longrightarrow}R^n \longrightarrow E \longrightarrow 0.    
\end{equation}
Moreover, with conditions (i) and (ii) above, after possible row and column operations, the matrix $\overline{\varphi}$ obtained from $\varphi$ modulo $(x_1, \ldots, x_s)$ has nonzero entries $x_{s+1}, \ldots, x_{d}$ which are concentrated all in one row or all in one column. Hence, there are two possible shapes of the matrix $\varphi$.  Namely, after row and column operations and a possible change of coordinates, the matrix $\varphi$ may be taken to have either the form $\varphi=\varphi_C$ or $\varphi = \varphi_R$ where

\begin{equation}\label{matrixshapes}
    \varphi_C \coloneq \left(\!\begin{array}{ccc|c}
   & & & *\\
   & \varphi'_C &   & \vdots \\
   & &  & *\\
   \hline
    * &\cdots& * & x_{s+1}\\
    \vdots & & \vdots & \vdots \\
    * &\cdots& * & x_{d}
\end{array}\!\right), \qquad  \qquad 
  \varphi_R \coloneq  \left(\!\begin{array}{ccc|ccc}
   & & &  *&\cdots & *\\
   & \varphi'_R&  &  \vdots & & \vdots \\
   & & & *&\cdots & *\\
   \hline
    * &\cdots& * &x_{s+1}&\cdots & x_d
\end{array}\!\right).
\end{equation}
In particular, the entries of the $(n-d+s)\times (n-e-1)$ submatrix $\varphi'_C$ and the $(n-1)\times (n-e-d+s)$ submatrix $\varphi'_R$, as well as all the $*$ entries, belong to the subring $A= k[x_1,\ldots,x_s] \subset R$. 
\end{rem}

As the indeterminates $x_{s+1},\ldots,x_d$ are either concentrated in a single column or a single row of $\varphi$, for the remainder of this paper we refer to these two cases as the \textit{column setting} and the \textit{row setting}, respectively. The next two sections are dedicated to a more thorough treatment of each case, highlighting substantial differences in the structure of the Rees algebra. Nevertheless, the two cases share several common features, which we discuss in the remainder of this section. 

\begin{rem}\label{s=d-1 remark}
Notice that if $s=d-1$, then the two possible shapes for the matrix $\varphi$ in (\ref{matrixshapes}) coincide and \Cref{rank1setting} recovers the setting of \cite[sec. 4]{CPW}. Moreover, following \Cref{G2P prop}, if $s=d-1$ and additionally $e=1$ and $d=3$, then \Cref{rank1setting} reduces to the setting of \cite{Nguyen17}. Likewise, when $s=d-1$ and also $e=1$ and $n=d+1$, one recovers the setting of \cite{Nguyen14}. 
\end{rem}

In particular, in our study of the Rees ring $\R(E)$ our main interest will be in the case when $s<d-1$, which will produce novel results, even in the rank-one case. We begin by identifying the \textit{non-linear type locus} of the module $E$, i.e. the set of prime ideals $\q\in \spec(R)$ for which $E_\q$ is not of linear type. We note that the following proposition holds more generally, outside of the assumptions of 
\Cref{rank1setting}. 

\begin{prop}\label{LTlocus}
Let $R$ be a Noetherian ring and $E$ a finitely generated $R$-module of rank $e$ and projective dimension one. If $E$ satisfies $G_s$ but not $G_{s+1}$, then $E_\q$ is of linear type for all primes $\q\in \spec(R)$ with $\q \notin V(\fitt_{s+e-1}(E))$.
\end{prop}

\begin{proof}
Notice that $V(\fitt_{s+e-1}(E)) = \{\q \in \spec(R) \,|\, \mu(E_\q) \geq s+e\}$ by \Cref{fitt}. Thus for any prime ideal $\q$ not contained in this set, we have that $\mu(E_\q)\leq s+e-1$. Since $E$ satisfies $G_s$, by transitivity of localization, for any $\p\subseteq \q$ it follows that $\mu\big((E_\q)_{\p R_\q}\big) = \mu (E_\p) \leq \min\{\hgt \p ,s\}+e-1$, so $E_\q$ satisfies $G_\infty$ and is hence of linear type by \cite[Prop. 3 and 4]{Avramov81}. 
\end{proof}

Thus, we are reduced to investigate the minimal primes of $\fitt_{s+e-1}(E)$. To this end, we recall a short lemma to assist us. 

\begin{lemma}[{\cite[2.4]{CPW}}]\label{GsFitt}
Let $R$ be a Noetherian ring and $E$ a finitely generated $R$-module of rank $e$ satisfying $G_s$ but not $G_{s+1}$. Then $\hgt \fitt_{s+e-2}(E)= \hgt \fitt_{s+e-1}(E) =s$.
\end{lemma}

Now returning to the assumptions of \Cref{rank1setting}, we show that $\fitt_{s+e-1}(E)$ has a \textit{unique} minimal prime. This fact is crucial for the following arguments and, unlike the previous items, we note that this does require the strength of \Cref{rank1setting}. We present some examples in \Cref{sec: Examples} of the behavior of $\R(E)$ outside of this setting, when this phenomenon does not occur.

\begin{prop}\label{minFitt}
 With the assumptions of \Cref{rank1setting}, $\p=(x_1,\ldots,x_{s})$ is the unique minimal prime of $\fitt_{s+e-1}(E) = I_{n-s-e+1}(\varphi)$.  
\end{prop}

\begin{proof}
Since $E$ satisfies $G_{s}$ but not $G_{s+1}$, we have $\hgt I_{n-s-e+2}(\varphi) = \hgt I_{n-s-e+1}(\varphi) = s$ by \Cref{GsFitt}. We now consider two cases, depending on the two possible shapes of the matrix $\varphi$ in (\ref{matrixshapes}).

For the column setting, let $\varphi =\varphi_C$ as in (\ref{matrixshapes}). Following the notation of \Cref{sec:preliminaries}, let $\varphi_1$ denote the $n \times (n-e-1)$ submatrix of $\varphi$ obtained by deleting the last column of $\varphi$. Then, the entries of $\varphi_1$ are linear forms in $A=k[x_1, \ldots,x_s]$, and also we have $I_{n-s-e+2}(\varphi) \subseteq I_{n-s-e+1}(\varphi_1) \subseteq (x_1, \ldots, x_s) = \p$. Since $\hgt I_{n-s-e+2}(\varphi) = s = \hgt \p$, it follows that $\hgt I_{n-s-e+1}(\varphi_1)= s$. Thus, $I_{n-s-e+1}(\varphi_1)$ is a $\p$-primary ideal in $A$, and so $\p$ is its only minimal prime in $A$ and hence in $R$. On the other hand, since $I_{n-s-e+1}(\varphi_1) \subseteq I_{n-s-e+1}(\varphi) \subseteq \p$ and $\hgt I_{n-s-e+1}(\varphi) =s$, it follows that  $\p$ is the unique minimal prime of $I_{n-s-e+1}(\varphi)$ too.

For the row setting, let $\varphi =\varphi_R$ as in (\ref{matrixshapes}). A similar argument as above yields the same result, by considering the submatrix obtained by deleting the last row of $\varphi$. 
\end{proof}

With the non-linear type locus of $E$ determined in \Cref{LTlocus}, we may now introduce our first description of the defining ideal $\J$ of the Rees ring $\R(E)$.

\begin{prop}\label{Jasat}
 With the assumptions of \Cref{rank1setting}, $\J$ is a prime ideal of height $n-e$. Moreover, we have that $\J=\L:\p^\infty$ where $\p= (x_1,\ldots,x_s)$ as before. 
\end{prop}

\begin{proof}
 The initial statement is clear as $\R(E)\cong R[T_1, \ldots, T_n]/\J$ and $\R(E)$ is a domain of dimension $d+e$. For the second statement, notice that $\L:(x_1,\ldots,x_s)^\infty \subseteq \J$ as $\J$ is prime, $\L\subseteq \J$ and $(x_1,\ldots,x_s)\nsubseteq \J$.  To prove the reverse containment, consider the quotient $\A = \J/\L$. It suffices to show that $\A$ is annihilated by some power of $\p= (x_1,\ldots,x_s)$. From \Cref{LTlocus} and \Cref{minFitt}, it follows that $\A$ is supported only at primes containing $\p$, hence $\p$ is the only minimal prime of $\ann\A$. Thus $\p = \sqrt{\ann \A}$, and so indeed a power of $\p$ annihilates $\A$.
\end{proof}

The description of $\J$ in \Cref{Jasat} will be crucial to our study. Our search for an explicit generating set of $\J$ now begins by extracting information from a Jacobian dual matrix $B(\varphi)$ of $\varphi$. We note, however, as there are two possible forms of the presentation $\varphi$ in (\ref{matrixshapes}), there are two possible shapes of its Jacobian dual matrix, with respect to $x_1,\ldots,x_d$. Namely $B(\varphi)$ coincides with either $B(\varphi_R)$ or $B(\varphi_C)$ where

\begin{equation}\label{JD rank1 shapes}
B(\varphi_C) = \left(\!\begin{array}{ccc|c}
   & &    & \bullet            \\
   & B'_C &  &       \vdots      \\
   & &      &       \bullet      \\
\hline
    0&\cdots&0    &  T_{n-d+s+1}  \\
    \vdots & &\vdots&  \vdots           \\
    0&\cdots&0       &T_n   
\end{array}\!\right),
\qquad \quad
B(\varphi_R) =  \left(\!\begin{array}{ccc|ccc}
& &                 & & & \\
   & B'_R &              &  & \psi &\\
   & &                  & &  &  \\
   \hline
    0&\cdots&0          & T_n\\
    \vdots & &\vdots    && \ddots         \\
    0&\cdots&0          &&&T_n
\end{array}\!\right)
\end{equation} 
which can be seen from (\ref{JD equation}) and the two possible shapes of $\varphi$ in (\ref{matrixshapes}). 

Here $B'_C$ is an $s\times (n-e-1)$ matrix with linear entries in $k[T_1,\ldots,T_n]$, and the $\bullet$ entries of $B(\varphi_C)$ belong to $k[T_{1},\ldots,T_{n-d+s}]$. Moreover, $B'_R$ is an $s\times (n-e-d+s)$ matrix and $\psi$ is an $s\times (d-s)$ matrix, both consisting of linear entries in $k[T_1,\ldots,T_n]$. 
Additionally, the bottom right block of $B(\varphi_R)$ is $T_n\cdot I_{d-s}$ where $I_{d-s}$ is the $(d-s)\times (d-s)$ identity matrix. 

\begin{rem}\label{s=d-1 JD remark}
Notice that if $s=d-1$, then the two possible shapes for $B(\varphi)$ in (\ref{JD rank1 shapes}) coincide with 
$$ B(\varphi) =  \left(\!\begin{array}{ccc|c}
   & &    & \bullet            \\
   & B' &  &       \vdots      \\
   & &      &       \bullet      \\
\hline
    0&\cdots&0    &  T_{n}  \\
\end{array}\!\right) $$
and one has $[x_1\ldots x_{d}]\cdot B'= [\ell_1\ldots\ell_{n-e-1}]$, as observed in \cite{CPW}. Moreover, in this case, the matrix $B'$ completely determines the nonlinear equations of the Rees algebra, which coincide with the \textit{fiber equations} of $\R(E)$ \cite[4.11, 4.12]{CPW}. In the proceeding sections, we will see that this is not always true when $s<d-1$. 
\end{rem}

As we will see in \Cref{sec:column section} and \Cref{sec:row section}, the two possible shapes of $\varphi$ and $B(\varphi)$ produce two very different generating sets for the defining ideal $\J$ of $\R(E)$ (compare \Cref{Defining Ideal Column Case} and \Cref{Defining Ideal Row Case}), yielding very different behaviors in terms of the fiber type property and the analytic spread of $E$ (compare \Cref{F(I)defideal Column} and \Cref{F(I)defideal Row}). To understand the differences between the two cases, it will be useful to relate the submatrices $B'_C$ and $B'_R$ of (\ref{JD rank1 shapes}) to the submatrices $\varphi_i$ and the ideals $\L_i$ introduced in \Cref{Strip Setting} and \Cref{LiJinotation}. Notice that $B'_C$ satisfies the equation
\begin{equation} \label{B'C matrix equation}
    [T_1 \ldots T_n] \cdot \varphi_1= [\ell_1\ldots\ell_{n-e-1}] =  [x_1\ldots x_{s}]\cdot B'_C
\end{equation}
and so $B'_C$ is precisely the Jacobian dual of the matrix $\varphi_1$, with respect to the sequence $x_1, \ldots, x_s$. Similarly, with $B'_R$ we have
\begin{equation} \label{B'R matrix equation}
   [T_1 \ldots T_n] \cdot \varphi_{d-s} = [\ell_1\ldots\ell_{n-e-(d-s)}] =  [x_1\ldots x_{s}]\cdot B'_R
\end{equation}
and so $B'_R$ is the Jacobian dual of the matrix $\varphi_{d-s}$ with respect to the sequence $x_1, \ldots, x_s$.

Recall that $\L=(\ell_1,\ldots,\ell_{n-e})$, where $[\ell_1\ldots\ell_{n-e}]=[T_1 \ldots T_n] \cdot \varphi$, is the defining ideal of $\S(E)$. Notice that with (\ref{B'C matrix equation}) and (\ref{B'R matrix equation}), depending on whether $\varphi= \varphi_C$ or $\varphi= \varphi_R$, we have $\L +I_{s}(B'_C) \subseteq \J$ or $\L +I_{s}(B'_R) \subseteq \J$ due to Cramer's rule, since $\J$ is a prime ideal. With this, we introduce the following result, which will be a crucial technical step towards determining the defining ideal $\J$ of $\R(E)$.

\begin{prop}\label{B'hgt}
With the assumptions of \Cref{rank1setting} and $B'_C$ and $B'_R$ as in (\ref{JD rank1 shapes}), the following hold.
\begin{itemize}
    \item[{\rm(a)}] If $\varphi=\varphi_C$, then $I_{s}(B'_C)$ is a Cohen-Macaulay prime ideal of height $n-s-e$.
    \item[{\rm(b)}] If $\varphi=\varphi_R$, then $I_{s}(B'_R)$ is a Cohen-Macaulay prime ideal of height $n-e-d+1$. 
\end{itemize}    
\end{prop}

\begin{proof}
By (\ref{JD rank1 shapes}), $B'_C$ is an $s \times (n-e-1)$ matrix with linear entries in $k[T_1,\ldots,T_n]$. Moreover,  $[\ell_1\ldots\ell_{n-e-1}] = [x_1\ldots x_{s}]\cdot B'_C$ by (\ref{B'C matrix equation}). Now, as $\L:(x_1,\ldots,x_{s})^\infty$ is a prime ideal of height $n-e$ by \Cref{Jasat}, from \cite[2.2]{BM16} it follows that $(\ell_1,\ldots,\ell_{n-e-1}):(x_1,\ldots,x_{s})^\infty$ is a prime ideal of height $n-e-1$. Hence, \cite[2.4]{BM16} implies that $I_{s}(B'_C)$ is a prime ideal with $\hgt I_{s}(B'_C) = (n-e-1)-s+1 = n-e-s$. Since this is the maximal possible height by \cite[Thm.~1]{EN62}, the Cohen-Macaulayness of $I_{s}(B'_C)$ then follows from \cite[A2.13]{Eisenbud}. This proves (a).

For (b), notice that according to (\ref{JD rank1 shapes}) and (\ref{B'R matrix equation}), $B'_R$ is an $s \times (n-e-d+s)$ matrix with linear entries in $k[T_1,\ldots,T_n]$ and $[\ell_1\ldots\ell_{n-e-d+s}] = [x_1\ldots x_{s}]\cdot B'_R$. As $\L:(x_1,\ldots,x_{s})^\infty$ is a prime ideal of height $n-e$, by successively applying \cite[2.2]{BM16} $d-s$ times, we obtain that $(\ell_1,\ldots,\ell_{n-e-d+s}):(x_1,\ldots,x_{s})^\infty$ is a prime ideal of height $n-e-d+s$. Thus, $I_{s}(B'_R)$ is a prime ideal with $\hgt I_{s}(B'_R) = (n-e-d+s)-s+1 = n-e-d+1.$ 
Again, this is the maximal possible height by \cite[Thm.~1]{EN62}, hence $I_{s}(B'_R)$ is Cohen-Macaulay by \cite[A2.13]{Eisenbud}. 
\end{proof}

In the next sections, we will combine the height calculation of \Cref{B'hgt} with an argument allowing us to determine the defining ideal of $\R(E)$ from that of $\R(E_1)$, where $E=\coker \varphi_1$ as in \Cref{Strip Setting}. 
The following proposition guarantees that the assumptions of \Cref{rank1setting} are preserved when passing from $E$ to $E_1$, in either case of \Cref{rank1shapes}.

\begin{prop}\label{Gs not Gs+1}
The module $E_1$ is an $R$-module of projective dimension one satisfying $G_s$, but not $G_{s+1}$.    
\end{prop}

\begin{proof}
From \Cref{Gs Right to Left}, we know that $E_1$ satisfies $G_s$, hence we only need to show that $E_1$ does not satisfy $G_{s+1}$. As $E_1$ has rank $e+1$, it suffices to show that $\hgt \fitt_{s+e}(E_1)\leq s$. We distinguish two cases, depending on the two possible shapes for the matrix $\varphi$ as in (\ref{matrixshapes}).

In the column setting, let $\varphi=\varphi_C$ as in (\ref{matrixshapes}) and write $\fitt_{s+e}(E_1)= I_{n-s-e}(\varphi_1)$, with $\varphi_1$ as in \Cref{Strip Setting}. Notice that $n-s-e \geq d-s \geq 1$, as $n\geq d+e$ and $s\geq d-1$ by assumption. Thus, $I_{n-s-e}(\varphi_1)$ is not the unit ideal and is hence contained in $(x_1,\ldots,x_s)$, as the entries of $\varphi_1$ belong to $A=k[x_1, \ldots, x_s]$. Thus, $\hgt \fitt_{s+e}(E_1)\leq s$ and the claim follows.

In the row setting, let $\varphi=\varphi_R$ as in (\ref{matrixshapes}). It is enough to show that $E_{d-s}$ does not satisfy $G_{s+1}$, since then neither does $E_1$, by \Cref{Gs Right to Left}. To prove this, notice that $E_{d-s}$ is an $R$-module of rank $e+d-s$ and $\fitt_{d+e-1}(E_{d-s}) = I_{n-d-e+1}(\varphi_{d-s})$. Similarly to the previous case, it follows that $I_{n-d-e+1}(\varphi_{d-s}) \subseteq (x_1,\ldots,x_s)$, as the entries of $\varphi_{d-s}$ are in $A$ and this is not the unit ideal. Thus, $\hgt \fitt_{d+e-1}(E_{d-s}) \leq s$ and so $E_{d-s}$ does not satisfy $G_{s+1}$, as required.
\end{proof}

\begin{rem}
    The proof of \Cref{Gs not Gs+1} shows that, in the row setting when $\varphi = \varphi_R$ as in (\ref{matrixshapes}), actually all modules $E_i$ for $0 \leq i \leq d-s$ satisfy $G_s$ but not $G_{s+1}$. That is, the maximal index $s$ such that $E_i$ satisfies $G_s$ is constant in this range. 
\end{rem}


\section{The column case}\label{sec:column section}

Throughout this section, we let $R$ and $E$ be as in \Cref{rank1setting} and we assume that the presentation matrix $\varphi$ of $E$ is in column form, $\varphi=\varphi_C$ as in (\ref{matrixshapes}). 
Our primary task is to determine the defining ideal $\J$ of the Rees ring $\R(E)$ through means of successive approximations of this algebra, as described in \Cref{sec:column stripping}. As we will see, the fact that a subset of variables is concentrated in one column of the presentation matrix $\varphi$ makes this setting quite amenable to this approximation method. Indeed, the first algebra $\R(E_1)$ alone will provide enough information to determine $\J$. Our main setting throughout is as follows.

\begin{set}\label{Column Case Setting}
Adopt the assumptions of \Cref{rank1setting} and assume that, after the appropriate change of coordinates and row and column operations, the presentation matrix $\varphi$ has the shape $\varphi =\varphi_C$ in (\ref{matrixshapes}).
\end{set}

We omit the subscript in $\varphi_{C}$ and instead write $\varphi$ throughout to simplify the notation. With this, the matrix $\varphi$ and its Jacobian dual, with respect to $x_1,\ldots,x_d$, are

\begin{equation}\label{Column Section - phi and B(phi)}
\varphi =  \left(\!\begin{array}{ccc|c}
   & & & *\\
   & \varphi' &   & \vdots \\
   & &  & *\\
   \hline
    * &\cdots& * & x_{s+1}\\
    \vdots & & \vdots & \vdots \\
    * &\cdots& * & x_{d}
\end{array}\!\right) \qquad \text{and} \qquad 
B(\varphi)=\left(\!\begin{array}{ccc|c}
   & &    & \bullet            \\
   & B' &  &       \vdots      \\
   & &      &       \bullet      \\
\hline
    0&\cdots&0    &  T_{n-d+s+1}  \\
    \vdots & &\vdots&  \vdots           \\
    0&\cdots&0       &T_n   
\end{array}\!\right)
\end{equation}
where the entries of $\varphi'$ and the $*$ entries belong to the subring $A=k[x_1,\ldots,x_s]$, and the entries of $B(\varphi)$ belong to $k[T_1,\ldots,T_n]$. We also recall that 
\begin{equation}\label{Column Case JD equation}
    [T_1\ldots T_n]\cdot \varphi = [\ell_1\ldots\ell_{n-e}] = [x_1\ldots x_d] \cdot B(\varphi)
\end{equation}
where $\L=(\ell_1,\ldots,\ell_{n-e})$ is the defining ideal of the symmetric algebra $\S(E)$, as in \Cref{sec:preliminaries}.

As noted, we aim to approximate, in a sense, the Rees ring $\R(E)$ with a well-understood algebra surjecting onto it. By employing the technique of successive approximations, we take this algebra to be the Rees ring of $E_1$, where $E_1$ is the $R$-module defined in \Cref{sec:column stripping}. We begin by making some observations on the module $E_1$, and then we will consider its Rees ring.

\begin{prop}\label{Column Setting - phi_1 and E_1 prop}
With the assumptions of \Cref{Column Case Setting}, let $\varphi_1$ denote the submatrix of $\varphi$ as in \Cref{sec:column stripping}, and  let $E_1$ denote the cokernel of $\varphi_1$. The module $E_1$ has projective dimension one and satisfies $G_s$, but not $G_{s+1}$, both as an $R$-module and as an $A$-module, where $A=k[x_1,\ldots, x_s]$. 
\end{prop}

\begin{proof}
 The first claim follows from \Cref{Gs not Gs+1}. Notice that since $\varphi$ has the shape in (\ref{Column Section - phi and B(phi)}), the entries of $\varphi_1$ belong to the subring $A$, hence $E_1$ is an $A$-module of projective dimension one. Moreover, since $\spec(A)\subset \spec(R)$, $E_1$ satisfies $G_s$, but not $G_{s+1}$, as an $A$-module as well.
\end{proof}

With this, we may describe the Rees ring of $E_1$ as both an $A$-module and an $R$-module. 

\begin{prop}\label{R(E1) column case}
The Rees algebra of $E_1$, as an $R$-module, is $\R(E_1)\cong R[T_1,\ldots,T_n]/\J_1$ where 
$$\J_1 = \L_1:(x_1,\ldots, x_s) = \L_1 +I_s(B')$$
and $\L_1= (\ell_1,\ldots, \ell_{n-e-1})$. 
Moreover, $\R(E_1)$ is Cohen-Macaulay, both as an $R$-module and as an $A$-module.
\end{prop}

\begin{proof}
Notice that the submatrix $B'$ is precisely the Jacobian dual of $\varphi_1$ with respect to $x_1,\ldots,x_s$, as observed in (\ref{B'C matrix equation}). Moreover, by \Cref{Column Setting - phi_1 and E_1 prop}, $E_1$ satisfies $G_s$ as an $A$-module, where $A=k[x_1,\ldots,x_s]$. Hence, from \cite[4.11]{SUV03} it follows that the Rees ring of $E_1$, as an $A$-module, is $\R_{A}(E_1) \cong A[T_1,\ldots,T_n]/\J_1$ with $\J_1 =\L_1:(x_1,\ldots, x_s) = \L_1 +I_s(B')$ and moreover, this ring is Cohen-Macaulay. 

Viewing now $E_1$ as an $R$-module, it then follows that $\R(E_1)\cong R[T_1,\ldots,T_n]/\J_1$, with $\J_1$ extended to $R[T_1,\ldots,T_n]$. Moreover, $\R(E_1)$ is Cohen-Macaulay, as $x_{s+1}, \ldots, x_d$ is a regular sequence modulo $\J_1$.
\end{proof}

We remark that, although in the previous proof we include a description of the Rees algebra of $E_1$ over the simpler ring $A$, the description of $\R(E_1)$ with $E_1$ viewed as an $R$-module is required for our purposes. Indeed, we require a map of $R$-modules $E_1\rightarrow E$ to induce the map of $R$-algebras $\R(E_1)\rightarrow\R(E)$, as in \Cref{sec:column stripping}. Recall that the kernel of this map is $\J/\J_1$, and this is a prime ideal of height one.

\begin{lemma} \label{lemmaModuloJ1col}
With the isomorphism $\R(E_1) \cong R[T_1,\ldots,T_n]/\J_1$ in \Cref{R(E1) column case}, write $\overline{\,\cdot\,}$ to denote images modulo $\J_1$. We have the following.
    \begin{itemize}
        \item[{\rm (a)}] Letting $\p=(x_1,\ldots,x_s)$, $\,\overline{\p}$ is a Cohen-Macaulay prime $\R(E_1)$-ideal of height one.
        
        \item[{\rm (b)}] The $\R(E_1)$-ideal $\,\overline{\L+I_{s}(B')}$ is Cohen-Macaulay of height one. 
    \end{itemize}
\end{lemma}

\begin{proof}
Recall that $\R_{A}(E_1) \cong A[T_1,\ldots,T_n]/\J_1$, where $A=k[x_1,\ldots,x_s]$. As $\p = (x_1,\ldots,x_s)$ is the homogeneous maximal ideal of $A$, the special fiber ring of $E_1$, as an $A$-module, is 
$$\F_{A}(E_1) \cong \R_{A}(E_1)/\p \R_{A}(E_1) \cong k[T_1,\ldots,T_n]/I_s(B')$$
by \cite[4.11]{SUV03}, noting that $\L_1\subseteq \p$. Moreover, this is a Cohen-Macaulay domain of dimension $s+e$, following \Cref{B'hgt}(a). On the other hand, notice that $\R(E_1) \cong \R_A(E_1)[x_{s+1},\ldots,x_d]$ and so
$$\R(E_1)/\p \R(E_1) \cong \F_{A}(E_1)[x_{s+1},\ldots,x_d].$$
Hence this quotient is a Cohen-Macaulay domain of dimension $s+e+(d-s) =d+e$. Now since $E_1$ has rank $e+1$ by \Cref{Gs Right to Left}(a), it follows that $\dim \R(E_1) = d+e+1$. Thus $\overline{\p}$ is a prime ideal with $\hgt \overline{\p}=1$, which shows (a).

For (b), notice that $\overline{\L+I_{s}(B')} = \overline{(\ell_{n-e})}$, which follows immediately from the expression of $\J_1$ in \Cref{R(E1) column case}. Since $\ell_{n-e} \notin \J_1$, it follows that $\overline{\ell_{n-e}}\neq 0$. As $\R(E_1)$ is a domain, the ideal $\overline{(\ell_{n-e})}$ is thus generated by a regular element. Hence it is Cohen-Macaulay of height one, as claimed. 
\end{proof}

In \cite[4.11]{CPW} it was shown that when $s=d-1$, the defining ideal of $\R(E)$ is $\J= \L \colon (x_1, \ldots, x_{d-1}) = \L +I_{d-1}(B')$. We now prove that the same property holds for every $s \leq d-1$, in the present setting.

\begin{thm}\label{Defining Ideal Column Case}
    With $\varphi$ and $E$ as in \Cref{Column Case Setting}, the defining ideal $\J$ of $\R(E)$ is 
    $$\J= \L:\p = \L+I_{s}(B')$$
where $\p=(x_1,\ldots,x_{s})$. Moreover, $\R(E)$ is Cohen-Macaulay.
\end{thm}

\begin{proof}
    Recall from \Cref{Jasat} that $\J = \L:\p^\infty$. Hence, by Cramer's rule one has the containments $\L+I_s(B') \subseteq \L:\p \subseteq \J$, and so we only 
 need to show that $\J= \L+I_s(B')$. Moreover, since $\J_1 \subseteq \L+I_s(B') \subseteq \J$, it suffices to prove that $\overline{\J} = \overline{\L+I_s(B')} = \overline{(\ell_{n-e})}$.

    Since clearly $\overline{(\ell_{n-e})} \subseteq \overline{\J}$, it is enough to show that $\overline{(\ell_{n-e})}$ and $\overline{\J}$ agree locally at the associated primes of $\overline{(\ell_{n-e})}$. Recall from \Cref{lemmaModuloJ1col}(b) that this ideal is Cohen-Macaulay of height one, hence all of its associated primes are minimal of height one. Thus it is enough to prove the stronger statement that $\overline{(\ell_{n-e})}_\q = \overline{\J}_\q$ for any prime $\R(E_1)$-ideal $\q$ with height one. Recall that $\overline{\p}$ is one such prime, due to \Cref{lemmaModuloJ1col}(a). 
    
    First, notice that if $\q \neq \overline{\p}$, we have
    $$\overline{\J}_\q = \overline{\L}_\q : \overline{\p}_\q^\infty =\overline{\L}_\q : \R(E_1)_\q = \overline{\L}_\q =\overline{(\ell_{n-e})}_\q ,$$
    since $\overline{\L} =\overline{(\ell_{n-e})}$. 
    If instead $\q=\overline{\p}$, note that $\overline{\ell_{n-e}} \notin \overline{\p}$. Indeed, one has that $\overline{(\ell_{n-e})} + \overline{\p} = (\overline{\gamma}) +\overline{\p}$ where $\gamma =x_{s+1}T_{n-d+s+1} +\cdots + x_dT_n$ following (\ref{Column Case JD equation}). Since $\p = (x_1,\ldots,x_s)$ and $x_1,\ldots,x_s,\gamma$ is a regular sequence, it follows that $\overline{(\ell_{n-e})} + \overline{\p} \neq \overline{\p}$, and so $\overline{\ell_{n-e}} \notin \overline{\p}$ as claimed. Therefore, $\overline{(\ell_{n-e})}_{\overline{\p}} = \R(E_1)_{\overline{\p}}$. Moreover, since $\overline{(\ell_{n-e})}\subseteq \overline{\J}$, it follows that $\overline{\J}_{\overline{\p}} = \R(E_1)_{\overline{\p}}$ as well. Thus $\overline{\J} = \overline{(\ell_{n-e})}$, and so $\J= \L+I_s(B')$, as claimed.

As for the Cohen-Macaulayness of $\R(E)$, notice that we have shown that $\J = \J_1+(\ell_{n-e})$. As noted in \Cref{lemmaModuloJ1col}, the element $\ell_{n-e}$ is regular modulo $\J_1$. Moreover, by \Cref{R(E1) column case} we know that $\R(E_1)$ is Cohen-Macaulay, hence $\R(E)\cong \R(E_1)/(\overline{\ell_{n-e}})$ is Cohen-Macaulay as well. 
\end{proof}

With this, we may easily describe the ideal defining the special fiber ring $\F(E)$.

\begin{cor}\label{F(I)defideal Column}
With the assumptions of \Cref{Column Case Setting}, the special fiber ring of $E$ is 
$$\F(E) \cong k[T_1,\ldots,T_n]/I_{s}(B').$$ 
In particular, $\F(E)$ is a Cohen-Macaulay domain of dimension $s+e$, i.e. $E$ has analytic spread $\ell(E)=s+e$. 
\end{cor}

\begin{proof} 
Recall that the entries of $B'$ belong to $k[T_1, \ldots, T_n]$. From \Cref{Defining Ideal Column Case}, it follows that $\J+(x_1,\ldots,x_d) = I_{s}(B')+(x_1,\ldots,x_d)$, and so $\F(E) \cong k[T_1,\ldots,T_n]/I_{s}(B')$. Moreover, by \Cref{B'hgt}(a) we have that $I_{s}(B')$ is a Cohen-Macaulay prime ideal of $k[T_1,\ldots,T_n]$ with $\hgt I_{s}(B') = n-s-e$. Hence it follows that $\F(E)$ is a Cohen-Macaulay domain with $\ell(E)=\dim \F(E) = s+e$.
\end{proof}

\begin{rem}\label{Column case fiber type}
With the shape of the defining ideal $\J$ in \Cref{Defining Ideal Column Case}, we note that the module $E$ is of \textit{fiber type} by \Cref{F(I)defideal Column}. In other words, aside from the generators of $\L$, the equations of $\J$ belong to the subring $k[T_1,\ldots,T_n]$.    
\end{rem}

As another consequence of \Cref{Defining Ideal Column Case}, we may also describe the defining ideal $\J$ as a residual intersection. 

\begin{prop}\label{J RI prop - column case}
 With the assumptions of \Cref{Column Case Setting}, the defining ideal $\J$ of $\R(E)$ may be realized as 
 $$\J=\L:(x_1,\ldots,x_s,\gamma)$$
 where $\gamma =x_{s+1}T_{n-d+s+1} +\cdots + x_dT_n$, as in the proof of \Cref{Defining Ideal Column Case}. Moreover, this is an $(n-e)$-residual intersection of $(x_1,\ldots,x_s,\gamma)$.
\end{prop}

\begin{proof}
Recall from the proof of \Cref{Defining Ideal Column Case} that $\L \subseteq (x_1,\ldots,x_s,\gamma)$. Moreover, notice that we have the containment
\begin{equation}\label{RI containment - column case}
 \L:(x_1,\ldots,x_s,\gamma)\subseteq \L:(x_1,\ldots,x_s) =\J
\end{equation}
where the equality follows from \Cref{Defining Ideal Column Case}. Cramer's rule then implies that $I_s(B')= I_{s+1}(B'') \subseteq \L:(x_1,\ldots,x_s,\gamma)$, where $B''$ is the matrix
\begin{equation}\label{B'' matrix - Column case}
B''=  \left(\!\begin{array}{ccc|c}
   & &    & \bullet            \\
   & B' &  &       \vdots      \\
   & &      &       \bullet      \\
\hline
    0&\cdots&0    &  1  \\ 
\end{array}\!\right)   
\end{equation}
satisfying the matrix equation 
\begin{equation}\label{B'' transition equation - Column Case}
[\ell_1 \ldots \ell_{n-e}] = [x_1\ldots x_s \,\, \gamma] \cdot B''.    
\end{equation}
Thus, $\L + I_s(B')\subseteq \L:(x_1,\ldots,x_s,\gamma)$, hence by \Cref{Defining Ideal Column Case} we have $\J= \L:(x_1,\ldots,x_s,\gamma)$.  Moreover, as $\hgt \J =n-e$ and $\L=(\ell_1,\ldots,\ell_{n-e}) \subseteq (x_1,\ldots,x_s,\gamma)$, we conclude that $\J$ is an $(n-e)$-residual intersection of $(x_1,\ldots,x_s,\gamma)$.    
\end{proof}

\begin{rem}\label{Column case - RI remark}
As noted in the proof of \Cref{Defining Ideal Column Case}, we have that $x_1,\ldots,x_s,\gamma$ is a regular sequence. Hence by \Cref{J RI prop - column case}, $\J$ is a residual intersection of a complete intersection ideal. The proof of the main result of \cite{CPW} relied on a similar observation, and then used \Cref{thmResInt} to obtain a generating set of $\J$. We remark that a similar path could be followed to give an alternative proof of \Cref{Defining Ideal Column Case}, however the method of successive approximations of $\R(E)$ allows for a much quicker proof, and gives that $\J$ is a residual intersection as a byproduct. Whereas the methods of \cite{CPW} could have been applied in the present setting, we will see in the proceeding section that they fail for the row case; see \Cref{Row Case - RI remark}.
\end{rem}

As noted in \Cref{s=d-1 remark}, the column setting and the row setting of \Cref{rank1shapes} coincide when $s=d-1$, in which case one has the main result of \cite{CPW}, which shows that the defining ideal of $\R(E)$ is $\J=\L + I_{d-1}(B')$, similarly to \Cref{Defining Ideal Column Case}. In the next section, we show that in the row setting, i.e. when $\varphi= \varphi_R$ in (\ref{matrixshapes}), the ideal $\L + I_{s}(B')$ will be \textit{strictly} contained in $\J$, when $s<d-1$. Moreover, unlike in \Cref{Defining Ideal Column Case}, the module $E$ will not be of fiber type.


\section{The row case}\label{sec:row section}

In this section, we describe the defining ideal $\J$ of $\R(E)$ with $E$ as in \Cref{rank1setting}, assuming that the presentation matrix $\varphi$ of $E$ is in row form $\varphi=\varphi_R$ as in (\ref{matrixshapes}). As before, we employ the method of successive approximations to study $\R(E)$. The key difference here, opposed to the case of \Cref{sec:column section}, is that this method becomes part of an induction argument to determine the defining ideal of $\R(E)$. Our setting throughout is as follows.

\begin{set}\label{Row Case Setting}
Adopt the assumptions of \Cref{rank1setting} and assume that, after the appropriate change of coordinates, and row and column operations, $\varphi$ has the shape $\varphi =\varphi_R$ in (\ref{matrixshapes}).
\end{set}

As in the previous section, we omit the subscript and write $\varphi$ for $\varphi_R$ throughout to simplify the notation. With this, the matrix $\varphi$ and its Jacobian dual, with respect to $x_1,\ldots,x_d$, are
\begin{equation}\label{Row Section - phi and B(phi)}
 \varphi = \left(\!\begin{array}{ccc|ccc}
   & & &  *&\cdots & *\\
   & \varphi'&  &  \vdots & & \vdots \\
   & & & *&\cdots & *\\
   \hline
    * &\cdots& * &x_{s+1}&\cdots & x_d
\end{array}\!\right) 
\qquad \text{and} \qquad
B(\varphi)=  \left(\!\begin{array}{ccc|ccc}
& &                 & & & \\
   & B' &              &  & \psi &\\
   & &                  & &  &  \\
   \hline
    0&\cdots&0          & T_n\\
    \vdots & &\vdots    && \ddots         \\
    0&\cdots&0          &&&T_n 
\end{array}\!\right)
\end{equation}
\smallskip
as in (\ref{matrixshapes}) and (\ref{JD rank1 shapes}). Along with the Jacobian dual $B(\varphi)$, consider the matrix 
\begin{equation}\label{Row Section - B'' and C}
C =  \left(\!\begin{array}{ccc|ccc}
& &                 & & & \\
   & B' &              &  & \psi &\\
   & &                  & &  &  \\
   \hline
    0&\cdots&0          & x_{s+1}& \cdots &x_d
\end{array}\!\right)
\end{equation}
where $B'$ and $\psi$ are the submatrices of $B(\varphi)$ in (\ref{Row Section - phi and B(phi)}).

\begin{rem}\label{B' and C matrices}
As noted prior to \Cref{B'hgt}, recall that $B' = B(\varphi_{d-s})$, the Jacobian dual of the submatrix $\varphi_{d-s}$ with respect to $x_1,\ldots,x_s$, adopting the notation of \Cref{sec:column stripping}. Moreover, notice that the matrix $C$ is a transition matrix as well. Indeed, from (\ref{B'R matrix equation}) and the shape of $C$ above, we see that
\begin{equation}\label{Row Case - Transition}
[\ell_1 \ldots \ell_{n-e-(d-s)}] = [x_1\ldots x_s] \cdot B' \quad \quad \text{and} \quad \quad    [\ell_1 \ldots \ell_{n-e}] = [x_1 \ldots x_s \,\,\, T_n] \cdot C.
\end{equation}
where $[\ell_1 \ldots \ell_{n-e}] = [T_1 \ldots T_n] \cdot \varphi$ are the equations defining the symmetric algebra $\S(E)$.
\end{rem}

Notice that, by Cramer's rule, and since $\J$ is a prime ideal, we have that $\L+I_s(B')+I_{s+1}(C)\subseteq \J$. The following theorem, which we prove by induction in \Cref{sec:induction}, shows that this containment is actually an equality. 

\begin{thm}\label{Defining Ideal Row Case}
With the assumptions of \Cref{Row Case Setting} and $\varphi$ the presentation of $E$ as in (\ref{Row Section - phi and B(phi)}), the Rees algebra of $E$ is $\R(E) \cong R[T_1,\ldots,T_n]/\J$ with 
$$\J=\L:(x_1,\ldots,x_s)= \L+I_s(B')+I_{s+1}(C),$$ 
where $\L=(\ell_1,\ldots,\ell_{n-e})$, and $B'$ and $C$ are the matrices in (\ref{Row Section - phi and B(phi)}) and (\ref{Row Section - B'' and C}), satisfying the equations of (\ref{Row Case - Transition}). Moreover, $\R(E)$ is Cohen-Macaulay.
\end{thm}


\subsection{Proof by Induction}\label{sec:induction}

We devote this subsection to the proof of \Cref{Defining Ideal Row Case}. As such, we provide the induction setting below.

\begin{set}\label{Row Section - Induction Setting}
With the assumptions of \Cref{Row Case Setting}, we proceed by induction on the difference $d-s$ in order to prove \Cref{Defining Ideal Row Case}. We note that the initial case $d-s=1$ has been shown in \cite[3.9]{CPW}, observing that $I_{s+1}(C) \subseteq I_s(B')$ in this setting. (Alternatively, this follows from \Cref{Defining Ideal Column Case}, due to \Cref{s=d-1 remark}.) Hence we assume that $d-s\geq 2$ and the claim holds up to $d-s-1$. 
\end{set}

As in the previous section, we employ the technique of successive approximations to produce a ring surjecting onto $\R(E)$ so that the kernel of this map is a prime ideal of height one. As before, we take such a ring to be the Rees ring of $E_1=\coker \varphi_1$, where $\varphi_1$ and $E_1$ are as in \Cref{sec:column stripping}. Notice that from the shape of $\varphi$ in (\ref{Row Section - phi and B(phi)}), the entries of $\varphi_1$ belong to the subring $R'=k[x_1,\ldots,x_{d-1}]$. In particular, $E_1$ may be viewed as a module over this subring.

\begin{prop}\label{R(E1) row case}
The Rees algebra of $E_1$ is, as an $R$-module, $\R(E_1)\cong R[T_1,\ldots,T_n]/\J_1$ with 
$$\J_1 = \L_1:(x_1,\ldots, x_s) = \L_1 +I_s(B')+I_{s+1}(C_1)$$
where $\L_1= (\ell_1,\ldots, \ell_{n-e-1})$, $B'$ is the submatrix of $B(\varphi)$ in (\ref{Row Section - phi and B(phi)}), and $C_1$ is the submatrix of $C$ in (\ref{Row Section - B'' and C}) obtained by deleting its last column. Moreover, $\R(E_1)$ is a Cohen-Macaulay domain of dimension $d+e+1$.
\end{prop}

\begin{proof}
Recall that by \Cref{Gs not Gs+1}, $E_1$ satisfies $G_s$, but not $G_{s+1}$, as an $R$-module. Moreover, from the shape of $\varphi$ and its submatrix $\varphi_1$ in (\ref{Row Section - phi and B(phi)}), an argument similar to the one in \Cref{Column Setting - phi_1 and E_1 prop} shows that $E_1$ satisfies $G_s$, but not $G_{s+1}$, as an $R'$-module as well, where $R'=k[x_1,\ldots, x_{d-1}]$. Hence $E_1$ satisfies the assumptions of \Cref{Row Case Setting} as an $R'$-module. Moreover, as $R'=k[x_1,\ldots, x_{d-1}]$, we see that $\dim R' -s = d-s-1$. Hence by the induction hypothesis in \Cref{Row Section - Induction Setting}, the Rees ring of $E_1$, as an $R'$-module, is
$$\R_{R'}(E_1)\cong R'[T_1,\ldots,T_n]/\J_1$$
where 
$$\J_1= \L_1 + I_{s}(B')+I_{s+1}(C_1)$$
with $\L_1 = (\ell_1,\ldots,\ell_{n-e-1})$ 
and $[\ell_1 \ldots \ell_{n-e-1}] = [T_1 \ldots T_n] \cdot \varphi_1$. Indeed, 
observe that $B'$ and $C_1$ satisfy the matrix equations 
$$[\ell_1 \ldots \ell_{n-(e+1)-(d-s-1)}] = [x_1\ldots x_s] \cdot B' \quad \quad \text{and} \quad \quad    [\ell_1 \ldots \ell_{n-e-1}] = [x_1 \ldots x_s \,\,\, T_n] \cdot C_1$$
similar to those in (\ref{Row Case - Transition}), noting that $\rank E_1=e+1$ and $\dim R'=d-1$. Hence $n-(e+1) -(d-s-1) = n-e-(d-s)$ and the equations above agree with (\ref{Row Case - Transition}). Alternatively, again noting that $\dim R'-s=d-s-1$, from \Cref{B' and C matrices} the first matrix required here may be realized as the Jacobian dual of $(\varphi_1)_{d-s-1} = \varphi_{d-s}$, which is again $B'$ by \Cref{B' and C matrices}. 

As the ideal $\J_1$ above is the ideal defining the Rees ring of $E_1$ as an $R'$-module, its extension to $R[T_1,\ldots,T_n]$ is the defining ideal of the Rees ring of $E_1$ as an $R$-module. In other words, viewing $\J_1$ as an ideal of $R[T_1,\ldots,T_n]$, we have
$$\R(E_1)\cong R[T_1,\ldots,T_n]/\J_1 .$$
Finally, since $\R(E_1) \cong \R_{R'}(E_1)[x_d]$ and $\R_{R'}(E_1)$ is Cohen-Macaulay by the induction hypothesis, it follows that $\R(E_1)$ is a Cohen-Macaulay domain of dimension $d+e+1$.
\end{proof}

Similarly to the previous section, we regard $E_1$ as an $R$-module so as to consider the epimorphism of $R$-algebras $\R(E_1)\rightarrow \R(E)$, as outlined in \Cref{sec:preliminaries}. Notice that the kernel of this map is precisely $\J/\J_1$, which is a prime $\R(E_1)$-ideal of height one. Inspired by analogous constructions used in \cite{BM16,KPU11,Weaver23,Weaver24}, we now introduce certain ideals in $\R(E_1)$ to study this kernel. To this end, the following notation will be useful.

\begin{notat}\label{K notation}
Let $\overline{\,\cdot\,}$ denote images modulo $\J_1$ in $\R(E_1)$. Let $C'$ denote the $s \times(n-e-1)$ submatrix of $C_1$ obtained by deleting its last row (equivalently, $C'$ is obtained from $C$ by deleting its last row and column). Consider the $R[T_1,\ldots,T_n]$-ideal $\,\K \coloneq (\ell_1,\ldots,\ell_{n-e-1}) +I_s(C')+(T_n)$. 
\end{notat}

\begin{prop}\label{Kbar CM}
The $\R(E_1)$-ideal $\overline{\K}$ is a Cohen-Macaulay ideal of height one.
\end{prop}

\begin{proof}
From the description of $\J_1$ in \Cref{R(E1) row case}, we have the containment $\J_1\subseteq \K$, noting that $I_s(B')\subseteq I_s(C')$. Thus $\hgt \K \geq n-e$, as $\J_1$ is a prime ideal of height $n-e-1$ and $T_n\in \K\setminus \J_1$. We next show that $\hgt \K = n-e$. To this end, notice that, since $T_n\in \K$, there exists an $s\times (n-e-1)$ matrix $\widetilde{C'}$ with entries in $k[T_1,\ldots,T_{n-1}]$ and elements $\widetilde{\ell}_1,\ldots,\widetilde{\ell}_{n-e-1} \in R[T_1,\ldots,T_{n-1}]$ such that 
\begin{equation}\label{K tilde form}
\K = (\widetilde{\ell}_1,\ldots,\widetilde{\ell}_{n-e-1}) +I_s(\widetilde{C'})+(T_n)    
\end{equation}
and $[\widetilde{\ell}_1\ldots\widetilde{\ell}_{n-e-1}] = [x_1\ldots x_s] \cdot \widetilde{C'}$.  Since $\hgt \K \geq n-e$, from (\ref{K tilde form}) above it follows that $\hgt \big( (\widetilde{\ell}_1,\ldots,\widetilde{\ell}_{n-e-1}) +I_s(\widetilde{C'})\big) \geq n-e-1$, as this is an ideal of the subring $R[T_1,\ldots,T_{n-1}]$. However, since
$$(\widetilde{\ell}_1,\ldots,\widetilde{\ell}_{n-e-1}) +I_s(\widetilde{C'}) \subseteq(\widetilde{\ell}_1,\ldots,\widetilde{\ell}_{n-e-1}) : (x_1,\ldots,x_s)$$
it follows that $\hgt \big((\widetilde{\ell}_1,\ldots,\widetilde{\ell}_{n-e-1}) : (x_1,\ldots,x_s)\big) \geq n-e-1$. Thus this ideal is an $(n-e-1)$-residual intersection of $(x_1,\ldots,x_s)$, and so by \Cref{thmResInt} it follows that 
\begin{equation}\label{tilde colon}
(\widetilde{\ell}_1,\ldots,\widetilde{\ell}_{n-e-1}) +I_s(\widetilde{C'}) =(\widetilde{\ell}_1,\ldots,\widetilde{\ell}_{n-e-1}) : (x_1,\ldots,x_s)    
\end{equation}
and this is a Cohen-Macaulay ideal of height exactly $n-e-1$. As $T_n$ is regular modulo this ideal, we then have that $\K$ is Cohen-Macaulay of height $n-e$, as claimed. Thus $\overline{\K}$ is a Cohen-Macaulay $\R(E_1)$-ideal with $\hgt \overline{\K} =1$.
\end{proof}

With the Cohen-Macaulayness of $\overline{\K}$ established, we now show that this ideal is \textit{linked}, in the sense of \cite{Huneke82}, to a particular $\R(E_1)$-ideal. We begin with the following lemma.

\begin{lemma}\label{Is(C')barnonzero}
Writing $\K = (\widetilde{\ell}_1,\ldots,\widetilde{\ell}_{n-e-1}) +I_s(\widetilde{C'})+(T_n)$ as in the proof of \Cref{Kbar CM}, we have $\overline{I_s(\widetilde{C'})}\neq 0$, modulo $\J_1$.
\end{lemma}

\begin{proof}
First note that, by degree considerations, if $I_s(\widetilde{C'})\subseteq \J_1$, then it must be that $I_s(\widetilde{C'})\subseteq I_s(B')$. We show that this is impossible.  Recall from the proof of \Cref{Kbar CM} that the ideal $(\widetilde{\ell}_1,\ldots,\widetilde{\ell}_{n-e-1}) +I_s(\widetilde{C'})$ has height $n-e-1$. As $(\widetilde{\ell}_1,\ldots,\widetilde{\ell}_{n-e-1}) \subseteq (x_1,\ldots,x_s)$, it follows that $\hgt \big( (x_1,\ldots,x_s)+I_s(\widetilde{C'})\big) \geq n-e-1$. Since $I_s(\widetilde{C'})$ is an ideal of the subring $k[T_1,\ldots,T_{n-1}]$, this implies that $\hgt I_s(\widetilde{C'}) \geq n-e-s-1$. Thus if $I_s(\widetilde{C'})\subseteq I_s(B')$, then by \Cref{B'hgt} we have that $n-e-s-1 \leq n-e-d+1$, hence $s\geq d-2$. However, from \Cref{Row Section - Induction Setting} we have that $s\leq d-2$, so we only need to consider the case that $s=d-2$.

So, suppose that $I_s(\widetilde{C'})\subseteq I_s(B')$ where $s=d-2$, and note that then $d\geq 4$, as $s\geq 2$ by assumption. Write $\widetilde{B'}$ for the submatrix of $\widetilde{C'}$ consisting of the first $n-e-2$ columns, which coincides with the matrix obtained from $B'$ by extracting its terms containing $T_n$. Notice that
$$I_{d-2}(\widetilde{C'})\subseteq I_{d-2}(B') \subseteq I_{d-2}(B')+(T_n) = I_{d-2}(\widetilde{B'})+(T_n),$$
hence it follows that $I_{d-2}(\widetilde{C'}) \subseteq I_{d-2}(\widetilde{B'})$ as both matrices consist of entries in $k[T_1,\ldots,T_{n-1}]$. Thus we have that $I_{d-2}(\widetilde{C'}) = I_{d-2}(\widetilde{B'})$, as $\widetilde{B'}$ is a submatrix of $\widetilde{C'}$.

With this, we next show that $\widetilde{\ell}_{n-e-1}$ is a regular element modulo $(\widetilde{\ell}_1,\ldots,\widetilde{\ell}_{n-e-2}) +I_{d-2}(\widetilde{B'})$. Recall from the proof of \Cref{B'hgt}(b) that $(\ell_1,\ldots,\ell_{n-e-2}):(x_1,\ldots,x_{d-2})^\infty$ is a prime ideal of height $n-e-2$. Hence by \cite[2.4]{BM16} and (\ref{Row Case - Transition}), we see that
$$(\ell_1,\ldots,\ell_{n-e-2}):(x_1,\ldots,x_{d-2})^\infty = (\ell_1,\ldots,\ell_{n-e-2}):(x_1,\ldots,x_{d-2})= (\ell_1,\ldots,\ell_{n-e-2}) +I_{d-2}(B').$$
As this is a prime ideal not containing $T_n$, it follows that 
$$(\ell_1,\ldots,\ell_{n-e-2}) +I_{d-2}(B')+(T_n) =(\widetilde{\ell}_1,\ldots,\widetilde{\ell}_{n-e-2}) +I_{d-2}(\widetilde{B'})+(T_n) $$
and this ideal has height $n-e-1$. Thus we deduce that $\hgt\big((\widetilde{\ell}_1,\ldots,\widetilde{\ell}_{n-e-2}) +I_{d-2}(\widetilde{B'})\big) =n-e-2$ as this ideal belongs to $R[T_1,\ldots,T_{n-1}]$. As $(\widetilde{\ell}_1,\ldots,\widetilde{\ell}_{n-e-2}) +I_{d-2}(\widetilde{B'}) \subseteq (\widetilde{\ell}_1,\ldots,\widetilde{\ell}_{n-e-2}):(x_1,\ldots,x_{d-2})$, by height considerations and \Cref{thmResInt}, it follows that 
$$(\widetilde{\ell}_1,\ldots,\widetilde{\ell}_{n-e-2}):(x_1,\ldots,x_{d-2}) = (\widetilde{\ell}_1,\ldots,\widetilde{\ell}_{n-e-2}) +I_{d-2}(\widetilde{B'})$$
with height exactly $n-e-2$, and this ideal is Cohen-Macaulay. Now, since $I_{d-2}(\widetilde{C'}) = I_{d-2}(\widetilde{B'})$ and $(\widetilde{\ell}_1,\ldots,\widetilde{\ell}_{n-e-1}) +I_{d-2}(\widetilde{C'})$ is a Cohen-Macaulay ideal of height $n-e-1$, the height calculation above implies that $\widetilde{\ell}_{n-e-1}$ is a regular element modulo $(\widetilde{\ell}_1,\ldots,\widetilde{\ell}_{n-e-2}) +I_{d-2}(\widetilde{B'})$, as claimed.

Having shown that $\widetilde{\ell}_{n-e-1}$ is regular modulo $(\widetilde{\ell}_1,\ldots,\widetilde{\ell}_{n-e-2}) +I_{d-2}(\widetilde{B'})$, we now use this to reach the desired contradiction. Observe that, as $I_{d-2}(\widetilde{B'}) = I_{d-2}(\widetilde{C'})\neq 0$, there exists a nonzero minor of $\widetilde{B'}$ and, without loss of generality, we may assume this is the minor consisting of the first $d-2$ columns. Hence, there is a nonzero $(d-3)\times (d-3)$ minor $\delta$ within the first $d-3$ columns (recall that $d\geq 4$). Assume $\delta$ is obtained by deleting row $i$, for some $1\leq i\leq d-2$. Let $\Delta$ denote the minor of $\widetilde{C'}$ consisting of the first $d-3$ columns of $\widetilde{B'}$ and the last column of $\widetilde{C'}$. 

Notice that $x_i \Delta \in (x_1,\ldots,x_{d-2}) I_{d-2}(\widetilde{C'}) = (x_1,\ldots,x_{d-2}) I_{d-2}(\widetilde{B'})\subseteq (\widetilde{\ell}_1,\ldots,\widetilde{\ell}_{n-e-2})$, recalling that $I_{d-2}(\widetilde{C'}) = I_{d-2}(\widetilde{B'})$. Moreover, by Cramer's rule, modulo $(\widetilde{\ell}_1,\ldots,\widetilde{\ell}_{d-3})$ we have  $x_i\Delta \equiv (-1)^{i+d-2}\widetilde{\ell}_{n-e-1}\delta$ and so it follows that
$$\widetilde{\ell}_{n-e-1}\delta \subseteq (\widetilde{\ell}_1,\ldots,\widetilde{\ell}_{n-e-2})+I_{d-2}(\widetilde{B'}).$$
However, recall that $\widetilde{\ell}_{n-e-1}$ is regular modulo $(\widetilde{\ell}_1,\ldots,\widetilde{\ell}_{n-e-2})+I_{d-2}(\widetilde{B'})$. Hence it must be that 
$\delta \in (\widetilde{\ell}_1,\ldots,\widetilde{\ell}_{n-e-2})+I_{d-2}(\widetilde{B'})$. But this is impossible by degree considerations, and so this is a contradiction.
\end{proof}

\begin{prop}\label{colons}
    With $\K$ as in \Cref{K notation}, we have that $\overline{\K}$ and $\overline{(x_1,\ldots,x_s,T_n)}$ are linked through the regular element $\overline{T_n}$, namely $\overline{(x_1,\ldots,x_s,T_n)} = \overline{(T_n)}:\overline{\K}$ and $\overline{\K} = \overline{(T_n)}:\overline{(x_1,\ldots,x_s,T_n)}$. 
 In particular, the $\R(E_1)$-ideal $\overline{(x_1,\ldots,x_s,T_n)}$ is unmixed of height one.
\end{prop}

\begin{proof}
From the construction of $\K$, it is clear that we have the containment $\overline{(x_1,\ldots,x_s,T_n)}\, \overline{\K} \subseteq \overline{(T_n})$, hence $\overline{(x_1,\ldots,x_s,T_n)} \subseteq \overline{(T_n)}:\overline{\K}$ and also $ \overline{\K}\subseteq \overline{(T_n)}:\overline{(x_1,\ldots,x_s,T_n)}$. In order to conclude these containments are actually equalities, it suffices to show they agree locally at the associated primes of the smaller ideal. To this end, it suffices to show  that $\overline{\K}$ and $\overline{(x_1,\ldots,x_s,T_n)}$ have no common associated prime. 

Suppose that there is some $\q\in\spec(\R(E_1))$ with $\q \in \ass(\overline{\K})\bigcap \ass\big(\overline{(x_1,\ldots,x_s,T_n)}\big)$. From \Cref{Kbar CM}, we see that $\q\in \ass \overline{\K} = \min \overline{\K}$, hence $\hgt \q =1$. Moreover, as $\q \supseteq \overline{(x_1,\ldots,x_s,T_n)}\neq 0$, it follows that $\q$ is a minimal prime of this ideal and so $\hgt \overline{(x_1,\ldots,x_s,T_n)}=1$ as well. Thus $\q$ is an ideal of height one which contains $\overline{\K}+\overline{(x_1,\ldots,x_s,T_n)} = \overline{I_s(\widetilde{C'})} +\overline{(x_1,\ldots,x_s,T_n)}$. However, from (\ref{tilde colon}) and \Cref{Is(C')barnonzero} it follows that, modulo $\overline{I_s(\widetilde{C'})}\neq 0$, $\overline{(x_1,\ldots,x_s,T_n)}$ contains a regular element. Hence $\hgt\big(\overline{\K}+\overline{(x_1,\ldots,x_s,T_n)}\big) \geq 2$, noting that $\R(E_1)$ is a Cohen-Macaulay domain by \Cref{R(E1) row case}. However, this is a contradiction.

Now that $\overline{\K}$ and $\overline{(x_1,\ldots,x_s,T_n)}$ have been shown to be linked, the last statement follows from \cite[0.1]{Huneke82}.
\end{proof}

With the $\R(E_1)$-ideal $\overline{\K}$ and \Cref{colons}, we now introduce a divisorial ideal that will ultimately pave the path to the defining ideal $\J$ of $\R(E)$. As previously noted, similar constructions and techniques have been applied in \cite{BM16,KPU11,Weaver23,Weaver24}. 

\begin{notat}
Consider the divisorial ideal $\D = \frac{\overline{\ell_{n-e}}\, \overline{\K}}{\overline{T_n}}$ and note that, since $\ell_{n-e} \in  (x_1,\ldots,x_s,T_n)$, this is actually an $\R(E_1)$-ideal, by \Cref{colons}.
\end{notat}

Notice that, since $\R(E_1)$ is a domain and both $\overline{\ell_{n-e}}\neq 0$ and $\overline{T_n}\neq 0$, the ideals $\overline{\K}$ and $\D$ are isomorphic. In particular, $\D$ is a Cohen-Macaulay $\R(E_1)$-ideal of height one following \Cref{Kbar CM}. Also notice that $\D \subseteq \overline{\J}$ since $\D(\overline{T_n}) \subseteq (\overline{\ell_{n-e}}) \subseteq\overline{\J}$, the ideal $\overline{\J}$ is prime, and $\overline{T_n}\notin \J$. 

\begin{prop}\label{D = Candidate Ideal}
With $B'$ and $C$ the matrices in (\ref{Row Section - B'' and C}) and $\L = (\ell_1,\ldots,\ell_{n-e})$, we have the equality
$$\overline{\L+I_s(B')+I_{s+1}(C)} = \D$$
in $\R(E_1)$.
\end{prop}

\begin{proof}
Notice that, by \Cref{R(E1) row case}, we have  $\overline{\L+I_s(B')+I_{s+1}(C)} = (\overline{\ell_{n-e}})+ \overline{I_{s+1}(C)}$ and the only nonzero minors of $\overline{I_{s+1}(C)}$ are those involving the last column. With this, recall that $T_n\in \K$, hence 
\begin{equation}\label{ell equation}
 \overline{\ell_{n-e}} = \frac{\overline{\ell_{n-e}}\, \overline{T_n}}{\overline{T_n}} 
\end{equation}
 and so $ \overline{\ell_{n-e}} \in \D$. Moreover, to show that $\overline{I_{s+1}(C)} \subseteq \D$, we need only show that every $(s+1)\times(s+1)$ minor of $\overline{C}$ involving the last column is contained in $\D$, as the other minors vanish. Without loss of generality, let $M$ denote the submatrix consisting of the first $s$ columns and the last column of $C$. By the matrix equation in (\ref{Row Case - Transition}) and Cramer's rule, in $\R(E_1)$ one has the equality
\begin{equation}\label{Cramer}
\overline{\det M} \cdot \overline{T_n} = \overline{\ell_{n-e}} \cdot \overline{\det M'} 
\end{equation}
where $M'$ is the submatrix of $C'$ consisting of its first $s$ columns. Thus $\overline{\det M} =\frac{ \overline{\ell_{n-e}} \,\overline{\det M'}}{\overline{T_n}}\in \D$ as well, and so it follows that $\overline{\L+I_s(B')+I_{s+1}(C)} \subseteq \D$.

To prove that the reverse containment holds, notice that $\overline{\K} = \overline{I_s(C')}+(\overline{T_n})$. Now, (\ref{ell equation}) shows that $\frac{\overline{\ell_{n-e}}\, \overline{T_n}}{\overline{T_n}} = \overline{\ell_{n-e}} \in \overline{\L+I_s(B')+I_{s+1}(C)}$. Moreover, a similar argument using Cramer's rule and producing an equation similar to (\ref{Cramer}) proves that $\frac{\overline{\ell_{n-e}}\, \overline{I_s(C')}}{(\overline{T_n})} \subseteq \overline{\L+I_s(B')+I_{s+1}(C)}$ as well. 
\end{proof}

Notice that \Cref{D = Candidate Ideal} proves that the ideal $\D$ agrees with the candidate for the defining ideal in \Cref{Defining Ideal Row Case}, modulo $\J_1$ in $\R(E_1)$. Thus the induction argument will be complete once it has been shown that $\D$ also agrees with $\overline{\J}$, and that this is a Cohen-Macaulay ideal. With this, we are ready to prove \Cref{Defining Ideal Row Case}. 

\begin{proof}[Proof of \Cref{Defining Ideal Row Case}]
Notice that $\J_1\subseteq \L+I_s(B')+I_{s+1}(C)\subseteq \J$, where the first containment is a consequence of \Cref{R(E1) row case} and the second follows from Cramer's rule and the transition equations in (\ref{Row Case - Transition}). Thus it suffices to show that $\D= \overline{\J}$ in $\R(E_1)$, by \Cref{D = Candidate Ideal}. Moreover, with this containment, it is enough to prove that $\D$ and $\overline{\J}$ agree locally at the associated primes of $\D$. As noted, $\D$ and $\overline{\K}$ are isomorphic, hence by \Cref{Kbar CM} we only to show that $\D_\q = \overline{\J}_\q$ for any prime $\q$ of $\R(E_1)$ with $\hgt \q =1$. We consider the following cases.

First, suppose that $\q$ is a height one prime ideal in $\R(E_1)$ and $\overline{(x_1,\ldots,x_s)}\nsubseteq \q$. Recall from \Cref{Jasat} that $\J = \L:(x_1,\ldots,x_s)^\infty$, hence $\overline{\J}_\q = \overline{\L}_\q = (\overline{\ell_{n-e}})_\q$. Moreover, since $\overline{(x_1,\ldots,x_s, T_n)}\nsubseteq \q$ as well, it follows from \Cref{colons} that $\R(E_1)_\q = \overline{(T_n)_\q} :\overline{\K}_\q$, hence $\overline{\K}_\q \subseteq \overline{(T_n)}_\q$. However, recall that $T_n \in \K$, and so $\overline{\K}_\q =\overline{(T_n)}_\q$. Thus $\D_\q = \frac{(\overline{\ell_{n-e}})_\q\, \overline{\K}_\q}{\overline{(T_n)}_\q}=(\overline{\ell_{n-e}})_\q$ as well.

Before we consider the case that $\overline{(x_1,\ldots,x_s)}\subseteq \q$, we observe that from the shape of $\varphi$ in (\ref{Row Section - phi and B(phi)}) and the description of $\J_1$ in \Cref{R(E1) row case} it follows that
$$\J_1+(x_1,\ldots,x_s) = (x_{s+1}T_n,\ldots,x_{d-1}T_n) + I_{s}(B')+I_s(C_1) +(x_1,\ldots,x_s).$$
Thus any prime ideal of $\R(E_1)$ containing $\overline{(x_1,\ldots,x_s)}$ contains either $\overline{(T_n)}$ or $\overline{(x_{s+1},\ldots,x_{d-1})}$. In particular, we deduce that the set of minimal primes of $\overline{(x_1,\ldots,x_s)}$ is
\begin{equation}\label{min primes}
\MinPrimes \overline{(x_1,\ldots,x_s)} = \big\{\,\overline{(x_1,\ldots,x_{d-1})}\,\big\}\bigcup \MinPrimes \overline{(x_1,\ldots,x_s,T_n)},    
\end{equation}
noting that $\overline{(x_1,\ldots,x_{d-1})}$ is a prime $\R(E_1)$-ideal of height one. Indeed, this ideal defines the special fiber ring, a domain, of $E_1$ viewed as a module over $R'=k[x_1,\ldots,x_{d-1}]$, as in the proof of \Cref{R(E1) row case}. With this, suppose that $\q$ is a prime of $\R(E_1)$ with height one such that $\overline{(x_1,\ldots,x_s)}\subseteq \q$, and consider the following two cases of (\ref{min primes}) above. 

First, suppose that $\q = \overline{(x_1,\ldots,x_{d-1})}$, and notice that then $\overline{(x_1,\ldots,x_s,T_n)} \nsubseteq \q$. Thus by \Cref{colons} it follows that $\overline{\K}_\q =\overline{(T_n)}_\q$. On the other hand, observe that $\overline{\ell_{n-e}}\notin\q$. Indeed, we have that $(\overline{\ell_{n-e}}) + \q = (\overline{x_d}\,\overline{T_n})+\q$ and neither $\overline{T_n}$ nor $\overline{x_d}$ is contained in $\q$. Hence $\D_\q = \frac{(\overline{\ell_{n-e}})_\q\, \overline{\K}_\q}{\overline{(T_n)}_\q}=(\overline{\ell_{n-e}})_\q = \R(E_1)_\q$. Moreover, as $\D_\q\subseteq \overline{\J}_\q$, it follows that $\overline{\J}_\q= \R(E_1)_\q$ as well.

Lastly, assume that $\q \in \MinPrimes \overline{(x_1,\ldots,x_s,T_n)}$. Notice that $\overline{\K}\nsubseteq\q$ as $\overline{\K}$ and $\overline{(x_1,\ldots,x_s,T_n)}$ have no common associated prime, as noted in the proof of \Cref{colons}. Thus $\overline{\K}_\q = \R(E_1)_\q$ and so  \Cref{colons} implies that $\overline{(x_1,\ldots,x_s,T_n)}_\q = \overline{(T_n)}_\q$. Moreover, observe that $\ell_{n-e} \in (x_1,\ldots,x_s,T_n)$, whence $(\overline{\ell_{n-e})}_\q\subseteq \overline{(x_1,\ldots,x_s,T_n)}_\q = \overline{(T_n)}_\q$. We prove that this containment is an equality. Recall that $\J_1 \subset R'[T_1,\ldots,T_n]$ where $R'=k[x_1,\ldots,x_{d-1}]$, thus it follows that $\overline{x_d}\notin \q$, as $\hgt \q =1$. Hence, in the localization at $\q$, it becomes a unit. Moreover, since $\ell_{n-e} \in (x_1,\ldots,x_s,x_dT_n)$, it follows that $\overline{\ell_{n-e}}$ is a unit multiple of $\overline{T_n}$ locally, hence $(\overline{\ell_{n-e})}_\q = \overline{(T_n)}_\q$. Thus $\D_\q = \frac{(\overline{\ell_{n-e}})_\q\, \overline{\K}_\q}{\overline{(T_n)}_\q}=\overline{\K}_\q = \R(E_1)_\q$. Again noting that $\D_\q\subseteq \overline{\J}_\q$, it follows that $\overline{\J}_\q= \R(E_1)_\q$ as well. 

Now that it has been shown that $\D= \overline{\J}$, recall that this ideal is isomorphic to $\overline{\K}$, which is Cohen-Macaulay by \Cref{Kbar CM}. Thus $\overline{\J}$ is a Cohen-Macaulay $\R(E_1)$-ideal, and so $\R(E) \cong \R(E_1)/\overline{\J}$ is Cohen-Macaulay.
\end{proof}

Now that the defining ideal of Rees ring $\R(E)$ is understood, in the next subsection we proceed to analyze the special fiber ring $\F(E)$.


\subsection{The fiber ring}

With the assumptions of \Cref{Row Case Setting}, recall that the defining ideal of $\R(E)$ is identified in \cref{Defining Ideal Row Case}. From this, we can also determine the ideal defining the special fiber ring $\F(E)$.

\begin{cor}
 \label{F(I)defideal Row}
The special fiber ring of $E$ is $\F(E) \cong k[T_1,\ldots,T_n]/I_s(B')$. Moreover, $\F(E)$ is a Cohen-Macaulay domain with dimension $\ell(E) = d+e-1$. In particular, $E$ has maximal analytic spread.
\end{cor}

\begin{proof}
With the defining ideal $\J$ in \Cref{Defining Ideal Row Case}, we see that $\J+(x_1,\ldots,x_d) = I_s(B')+(x_1,\ldots,x_d)$ and the initial claim follows, noting that $B'$ consists of entries in $k[T_1,\ldots,T_n]$. The second assertion then follows from \Cref{B'hgt}.
\end{proof}

Notice that from \Cref{F(I)defideal Column} and \Cref{F(I)defideal Row}, the analytic spread of $E$ differs in the row setting and the column setting; however, they coincide when $s=d-1$, as expected from \Cref{s=d-1 remark}.

Like in the column case (see \Cref{J RI prop - column case}), in addition to the description provided in \Cref{Defining Ideal Row Case}, the defining ideal $\J$ of $\R(E)$ may also be realized as a residual intersection.  

\begin{prop}\label{J RI prop - row case}
 With the assumptions of \Cref{Row Case Setting}, the defining ideal $\J$ of $\R(E)$ may be realized as 
 $$\J=\L:(x_1,\ldots,x_s,x_{s+1}T_n,\ldots,x_d T_n).$$
 Moreover, this is an $(n-e)$-residual intersection.
\end{prop}

\begin{proof}
Notice that we have the containments 
\begin{equation}\label{RI containment}
  \L:(x_1,\ldots,x_s,x_{s+1}T_n,\ldots,x_d T_n)\subseteq \L:(x_1,\ldots,x_s) =\J
\end{equation}
where the equality follows from \Cref{Defining Ideal Row Case}. From Cramer's rule and (\ref{Row Case - Transition}), we have $I_{s+1}(C) \subseteq \L:(x_1,\ldots,x_s,T_n) \subseteq \L:(x_1,\ldots,x_s,x_{s+1}T_n,\ldots,x_d T_n)$. Moreover, using Cramer's rule again, we see $I_s(B')= I_d(B'') \subseteq \L:(x_1,\ldots,x_s,x_{s+1}T_n,\ldots,x_d T_n)$, where $B''$ is the matrix
\begin{equation}\label{B'' matrix}
B''=  \left(\!\begin{array}{ccc|ccc}
& &                 & & & \\
   & B' &              &  & \psi &\\
   & &                  & &  &  \\
   \hline
    0&\cdots&0          & 1\\
    \vdots & &\vdots    && \ddots         \\
    0&\cdots&0          &&&1 
\end{array}\!\right)    
\end{equation}
obtained by replacing the $T_n$ entries of the lower diagonal block of $B(\varphi)$ in (\ref{Row Section - phi and B(phi)}) with $1$s. Notice that this matrix satisfies the transition equation 
\begin{equation}\label{RI transition}
[\ell_1 \ldots \ell_{n-e}] = [x_1\ldots x_s \,\, x_{s+1}T_n\,\, \ldots \,\, x_dT_n] \cdot B''.    
\end{equation}
Thus $\L + I_s(B') +I_{s+1}(C) \subseteq \L:(x_1,\ldots,x_s,x_{s+1}T_n,\ldots,x_d T_n)$, and so by \Cref{Defining Ideal Row Case} we have $\J=\L:(x_1,\ldots,x_s,x_{s+1}T_n,\ldots,x_d T_n)$. Moreover, as $\hgt \J =n-e$ and $\L=(\ell_1,\ldots,\ell_{n-e})$ it follows that this is an $(n-e)$-residual intersection.
\end{proof}

\begin{rem}\label{Row Case - RI remark}
Although the defining ideal may be realized as a residual intersection, this description of $\J$ in not very useful in order to describe the properties of $\J$. Indeed, it is particularly ill-behaved, as the ideal $I=(x_1,\ldots,x_s,x_{s+1}T_n,\ldots,x_d T_n)$ is not unmixed unless $s=d-1$, in which case $I$ is a complete intersection and a generating set of $\J$ may be obtained from \Cref{thmResInt}. Moreover, in this case the residual intersection of \Cref{J RI prop - row case} coincides with the residual intersection of \Cref{J RI prop - column case}, and the main result of \cite{CPW} is recovered. If however $s\leq d-2$, then $I$ is not a complete intersection, and the techniques used in \cite{CPW} fail in the present setting, so the method of successive approximations must be applied instead; compare with \Cref{Column case - RI remark} in the column setting.
\end{rem}

With the bigrading on $R[T_1,\ldots,T_n]$ given by $\bideg x_i=(1,0)$ and $\bideg T_i = (0,1)$, observe that the equations of $I_{s+1}(C)$ in \Cref{Defining Ideal Row Case} have bidegree $(1,s)$. Hence it follows that the module $E$ is not of fiber type, once it has been shown that the equations of $I_{s+1}(C)$ are \textit{minimal} generators of $\J$ when $s\leq d-2$. We prove this in the following proposition, by exploiting the residual intersection property of $\J$.

\begin{prop}\label{fiber type prop}
With the assumptions of \Cref{Row Case Setting}, the module $E$ is of fiber type if and only if $s=d-1$.
\end{prop}

\begin{proof}
Recall from \Cref{Defining Ideal Row Case} that the defining ideal of $\R(E)$ is $\J = \L+I_s(B')+I_{s+1}(C)$. From bidegree considerations, it follows that the fiber type property is equivalent to the containment $I_{s+1}(C) \subseteq \L+I_s(B')$ or rather $\J = \L+I_s(B')$. Thus we have already seen that $E$ is of fiber type if $s=d-1$ within the initial case of the induction proof (see also \cite{CPW}), hence we only need to show the converse.

Suppose that $E$ is of fiber type, and so $\J = \L+I_s(B')$. Writing $I =(x_1,\ldots,x_s,x_{s+1}T_n,\ldots,x_d T_n)$, recall from \Cref{J RI prop - row case} that $\J$ may also be realized as the $(n-e)$-residual intersection $\J = \L:I$. 
We prove that the ideals $\L:I$ and $\L+I_s(B')$ can only coincide when $s=d-1$, by showing that otherwise these two descriptions of $\J$ provide two distinct values for the type of $\R(E)$, denoted $r(\R(E))$ \cite{BH93}. We will crucially use the fact that $r(\R(E))$ coincides with the minimal number of generators of the canonical module $\omega_{\R(E)}$ \cite[3.3.11]{BH93}.

Recall that $B'$ is the submatrix of the Jacobian dual $B(\varphi)$ satisfying the transition equation in (\ref{Row Case - Transition}). Thus, $B'$ is the Jacobian dual of the matrix $\varphi_{d-s}$ with respect to the sequence $x_1, \ldots, x_s$. Also, $\varphi_{d-s}$ presents $E_{d-s}$, and by \Cref{Gs Right to Left} $E_{d-s}$ is an $A$-module satisfying $G_s$, where $A=k[x_1,\ldots,x_s]$. 
Hence by \cite[4.11]{SUV03}, the defining ideal of $\R(E_{d-s})$ is $\L_{d-s}:(x_1,\ldots,x_s) = \L_{d-s} +I_s(B')$, where $\L_{d-s}=(\ell_1,\ldots,\ell_{n-e-(d-s)})$. Moreover, this is a Cohen-Macaulay prime ideal of height $n-e-(d-s)$. With this, we obtain that $\J = \L+I_s(B') = \J_{d-s} + (\ell_{n-e-(d-s)+1},\ldots, \ell_{n-e})$. Since $\J$ is a prime ideal of height $n-e$, it then follows that $\ell_{n-e-(d-s)+1},\ldots, \ell_{n-e}$ is a regular sequence modulo $\J_{d-s}$.

As $\J_{d-s} = (\ell_1,\ldots,\ell_{n-e-(d-s)}) +I_s(B')$ is Cohen-Macaulay of height $n-e-(d-s)$, the transition equation (\ref{Row Case - Transition}) and \Cref{B'hgt} show that the complex $\mathbb{T} = \mathbb{T}(\L_{d-s},B')$ of \cite{BKM90} is a minimal free resolution of $\R(E_{d-s})$. In particular, from \cite[3.5]{BKM90}, the last Betti number of $\mathbb{T}$ is $\binom{n-e-d+s}{s-1}$ and we note that this is precisely the type of $\R(E_{d-s})$ \cite[3.3.9]{BH93}. Moreover, as the type of a ring is unchanged modulo a regular sequence \cite[1.2.19]{BH93}, from the previous observation, it follows that $\mu(\omega_{\R(E)})=r(\R(E))=\binom{n-e-d+s}{s-1}$ as well.

On the other hand, as $\J=\L:I$ is a residual intersection, by \cite[5.1]{HU88} it follows that the canonical module $\omega_{\R(E)}$ can be realized as the symmetric power $\omega_{\R(E)} \cong \S_B^{n-e-s}(I/\L)$, where $B=R[T_1,\ldots,T_n]$. Notice that, due to the shape of $\varphi$ in (\ref{Row Section - phi and B(phi)}), we have that $\mu(I/\L)=s$. Indeed, modulo $\L$ the images of $x_{s+1}T_n,\ldots,x_d T_n$ are contained in $(x_1,\ldots,x_s)$. Therefore, we obtain that
$$\mu(\omega_{\R(E)}) = {\rm dim}_k\, \S_B^{n-e-s}(I/\L) \otimes k = {\rm dim}_k\, \S_k^{n-e-s}(k^s) =\binom{n-e-1}{s-1}.$$
Thus, if $\J=\L:I=\L+I_s(B')$ we must have that $\binom{n-e-d+s}{s-1} = \binom{n-e-1}{s-1}$, from which it follows that $s=d-1$.
\end{proof}


\section{Examples and Future Directions}\label{sec: Examples}

In this section, we explore various questions related to modifying the conditions of \Cref{rank1setting}. When these conditions are changed, do the results obtained in this paper remain the same? Additionally, what are some potential next steps that could be taken to further explore the study of defining equations of Rees algebras when a module or ideal only satisfies the condition $G_{s}$ but not $G_{s+1}$?

As noted in \Cref{minFitt} and \Cref{Jasat}, the fact that $\fitt_{s+e-1}(E) = I_{n-s-e+1}(\varphi)$ has a unique minimal prime in \Cref{rank1setting} is crucial to the proofs of \Cref{Defining Ideal Column Case} and \Cref{Defining Ideal Row Case}. Whereas \Cref{minFitt} shows that the assumptions of \Cref{rank1setting} are sufficient to ensure this phenomenon, one might ask if the assumptions of \Cref{rank1setting} are necessary. Moreover, one might ask what the structure of the defining ideal $\J$ is when $\fitt_{s+e-1}(E)$ has more than one minimal prime.

We present some examples here illustrating interesting phenomena related to these questions. Recall from \Cref{G2P prop} that if the rank of $E$ is $e=1$, then $E$ is isomorphic to a perfect $R$-ideal of grade two. Hence, for computational purposes, we provide examples in the case of these ideals, making use of the Hilbert-Burch theorem \cite[20.15]{Eisenbud} throughout. Each of the examples presented here was produced and verified through \textit{Macaulay2} \cite{Macaulay2}.

\begin{ex}\label{exampleRank2}
     Consider the following $6 \times 5$ matrix with entries in $R=\mathbb{Q}[x_1,x_2,x_3,x_4]$:
    $$\varphi = \left(\!\begin{array}{ccccc}
0&0&0&0&x_{2}\\
x_{2}&x_{1}+x_{2}&0&x_{1}+x_{2}&x_{1}\\
0&0&x_{3}&x_{3}&x_{4}\\
0&x_{2}&x_{1}+x_{2}&0&x_{1}+x_{2}\\
x_{4}&x_{3}+x_{4}&0&0&x_{3}\\
0&0&x_{4}&0&x_{1}
\end{array}\!\right)$$
and consider the ideal $I = I_{5}(\varphi)$. One has that $\hgt I=2$, hence $I$ is perfect of height two. Moreover, the ideal $I$ satisfies $G_{2}$ but not $G_{3}$. In this example, $\fitt_2(I) = I_{4}(\varphi)$ has two minimal prime ideals, $\left(x_{1},x_{2}\right)$ and $\left(x_{3},x_{4}\right)$. As $I$ satisfies assumptions (i) and (iii) of \Cref{rank1setting}, it follows from \Cref{minFitt} that modulo any set of two linear forms, $\varphi$ has rank at least $2$. 

With the bigrading on $R[T_1,\ldots,T_5]$ given by $\bideg x_i=(1,0)$ and $\bideg T_i = (0,1)$, the defining ideal $\J$ of $\R(I)$ consists of the equations of $\L = ([T_1\ldots T_5]\cdot \varphi)$ in bidegree $(1,1)$, one equation of bidegree $(1,3)$, two equations of bidegree $(2,2)$, and fiber equations with one of bidegree $(0,3)$, and four of bidegree $(0,4)$. Hence the shape of 
$\J$ does not agree with either of the forms in \Cref{Defining Ideal Column Case} or \Cref{Defining Ideal Row Case}, as it is minimally generated in different bidegrees.
\end{ex}

Whereas, in the example above, the ideal $\fitt_2(I)$ has two minimal primes, they are both complete intersections generated by $s=2$ many linear forms; this is comparable to the behavior discussed in \Cref{minFitt}. However, even this behavior is not guaranteed and, as the next example shows, this Fitting ideal may have multiple minimal primes of differing codimensions.

\begin{ex}\label{exampleRank3}
Consider the following $5 \times 4$ matrix with entries in $R=\mathbb{Q}[x_{1},x_{2},x_{3},x_{4}]$:
$$\varphi = \left(\!\begin{array}{cccc}
x_{1}-x_{2}&x_{2}&x_{2}&x_{1}\\
x_{2}&0&x_{2}&x_{1}\\
x_{1}+x_{2}&0&x_{2}&x_{1}\\
x_{4}&x_{1}&x_{3}&0\\
x_{1}&x_{3}&x_{1}&x_{4}
\end{array}\!\right)$$
and consider the ideal $I=I_4(\varphi)$. One has that $\hgt I =2$, hence $I$ is perfect of height two. Moreover, $I$ satisfies $G_{2}$ but not $G_{3}$. Here, $\fitt_2(I) = I_{3}(\varphi)$ has two minimal primes of different heights, namely $(x_{1},x_{2})$ and $(x_{1},x_{3},x_{4})$. As $I$ satisfies assumptions (i) and (iii) of \Cref{rank1setting}, it follows from \Cref{minFitt} that modulo any set of two linear forms, $\varphi$ has rank at least $2$. For instance, modulo $(x_{1},x_{2})$, $\varphi$ has rank $2$.

With the bigrading on $R[T_1,\ldots,T_5]$ given by $\bideg x_i=(1,0)$ and $\bideg T_i = (0,1)$, the defining ideal $\J$ of $\R(I)$ consists of the equations of $\L = ([T_1\ldots T_5]\cdot \varphi)$ in bidegree $(1,1)$, one equation of bidegree $(2,2)$, and one fiber equation with bidegree $(0,4)$. Hence the shape of 
$\J$ does not agree with either of the forms in \Cref{Defining Ideal Column Case} or \Cref{Defining Ideal Row Case}, as there is a minimal generator with bidegree $(2,2)$.
\end{ex}

In light of \Cref{exampleRank2} and \cref{exampleRank3}, we can see that the results obtained in this article are not generally true in the absence of condition (ii) of \Cref{rank1setting}. However, there are still examples for a module $E$ satisfying $G_s$, but not $G_{s+1}$, where $\fitt_{s+e-1}(E)$ has $\p = (x_1,\ldots,x_s)$ as its unique minimal prime, as in \Cref{minFitt}, even if $\varphi$ has rank greater than $1$ modulo this ideal.

\begin{ex}\label{exampleRank2Still1MinPrime}
Consider the following $5 \times 4$ matrix with entries in $R=\mathbb{Q}[x_1,x_2,x_3,x_4]$:
   $$\varphi = \left(\!\begin{array}{cccc}
x_{2}&0&x_{2}&0\\
x_{2}&x_{1}&x_{4}&x_{2}\\
0&x_{1}&x_{2}&x_{3}\\
0&x_{2}&x_{3}&x_{1}\\
x_{1}&x_{2}&x_{1}&x_{4}
\end{array}\!\right)$$
and consider the ideal $I= I_{4}(\varphi)$. This ideal is perfect of grade two and satisfies $G_{2}$ but not $G_{3}$. The ideal $I$ satisfies conditions (i) and (iii) of \Cref{rank1setting}, and the ideal $\fitt_2(I)$ has $(x_1,x_2)$ as its only minimal prime. However, the rank of $\varphi$ is 2 modulo $(x_1,x_2)$.

Moreover, \Cref{LTlocus} and the proof of \Cref{Jasat} imply that the defining ideal of $\R(I)$ is $\J = \L:\p^\infty$ where $\L = ([T_1\ldots T_5]\cdot \varphi)$ and $\p=(x_1,x_2)$. Computations through \textit{Macaulay2} \cite{Macaulay2} show that actually $\J=\L:\p$ and that $\J$ is minimally generated by the equations of $\L$ and one fiber equation of bidegree $(0,2)$, which may be taken as the determinant of a $2\times 2$ submatrix of $B(\varphi)$. In particular, $\J$ has the shape of the ideal in \Cref{Defining Ideal Column Case} without $\varphi$ or $B(\varphi)$ having the form in (\ref{Column Section - phi and B(phi)}), after a change of coordinates.
\end{ex}

As noted, the condition that $\fitt_{s+e-1}(E)$ has a unique minimal prime allows the defining ideal $\J$ of $\R(E)$ to be written as a saturation. Moreover, since the defining ideal $\J$ in \Cref{exampleRank2Still1MinPrime} resembles the behavior observed in \Cref{Defining Ideal Column Case}, one might ask the following question.

\begin{quest}\label{questionUniquePrime}
Can one characterize when $\fitt_{s+e-1}(E)$ has a unique minimal prime if $E$ is a module satisfying all of the assumptions of \Cref{rank1setting} except for condition (ii)?
\end{quest}

An answer to this question could possibly allow for a description of $\J$ in a more general setting than the assumptions of \Cref{rank1setting}. Additionally, \Cref{exampleRank2} and \Cref{exampleRank3} lead to the following more general question.

\begin{quest}\label{questionFittingPrimes}
What conditions can be placed, to allow one to determine the minimal primes of the Fitting ideals of $E$, for a module $E$ satisfying all conditions of \Cref{rank1setting} except for the rank condition (ii)?
\end{quest}

In \cite{DRS18}, the notion of the \textit{chaos invariant} was introduced to relate the minimal primes of Fitting ideals and the defining equations of the Rees ring, without relying on an assumption like condition (ii) of \Cref{rank1setting}. This invariant was introduced in the setting where $R = k[x_1,x_2,x_3]$ and $I$ is a linearly presented perfect ideal of height $2$ satisfying $G_{2}$ but not $G_{3}$. Perhaps a generalization of the chaos invariant to $R = k[x_1,\ldots,x_d]$ and a linearly presented $R$-module $E$ satisfying $G_{s}$ but not $G_{s+1}$ could assist in answering \Cref{questionUniquePrime} and \Cref{questionFittingPrimes}.

In addition to condition (ii) of \Cref{rank1setting}, we also remark that condition (i) is restrictive as well. The assumption that the module $E$ has a presentation matrix $\varphi$ consisting of linear entries was crucial for the arguments presented here and, aside from \cite{BM16,CHW08,Costantini, KPU11}, this condition has seldom been weakened in the context of modules of projective dimension one. However, it is curious if the techniques presented here can be combined with the techniques presented in \cite{BM16}, to study the case when condition (i) is relaxed. 

\begin{quest}\label{ALP question}
Can the equations defining $\R(E)$ be determined for a module $E$ satisfying the assumptions of \Cref{rank1setting} except for condition (i)?
\end{quest}

In particular, as a natural extension of the work in \cite{BM16,Costantini}, one might consider ideals and modules of projective dimension one not satisfying $G_d$, that have \textit{almost} linear presentation; that is, their presentation matrix $\varphi$ consists of linear entries, except for one column of entries of a higher degree.

\begin{ex}\label{exampleAlmostLinearlyPresented}
    Consider the following $5 \times 4$ matrix with entries in $R=\mathbb{Q}[x_{1},x_{2},x_{3},x_{4}]$:
    $$\varphi = \left(\!\begin{array}{cccc}
x_{1}^{2}&x_{1}&x_{2}&0\\
0&0&x_{1}&x_{1}\\
x_{2}^{2}&x_{2}&x_{1}&0\\
0&x_{1}&x_{2}&x_{2}\\
x_{2}^{2}&x_{1}&x_{3}&x_{4}
\end{array}\!\right)$$
and let $I = I_{4}(\varphi)$. The ideal $I$ is perfect of height two and satisfies $G_{2}$ but not $G_{3}$. Modulo $(x_{1},x_{2})$, the matrix $\varphi$ has rank $1$, and moreover $\fitt_2(I)$ has a unique minimal prime, namely $(x_{1},x_{2})$. 

By \Cref{LTlocus} and by repeating the proof of \Cref{Jasat}, it follows that the defining ideal $\J$ of $\R(E)$ is $\L:\p^\infty$ where $\L = ([T_1\ldots T_5]\cdot \varphi)$ and $\p=(x_1,x_2)$. However, computations through \textit{Macaulay2} \cite{Macaulay2} show that actually $\J=\L:\p^2 \neq \L:\p$, differing from the behavior in \Cref{Defining Ideal Column Case} or \Cref{Defining Ideal Row Case}, but similar to the behavior in \cite[3.6]{BM16}, as the entries in the nonlinear column of $\varphi$ have degree $2$.
\end{ex}


\end{document}